\documentclass[11pt,leqno]{amsart}
\usepackage{amsthm,amsfonts,amssymb,amsmath,oldgerm,color,hyperref}

\usepackage{mathtools}

\mathtoolsset{showonlyrefs}

\renewcommand\d{\partial}
\renewcommand\a{\alpha}
\renewcommand\b{\beta}

\newcommand\s{\sigma}
\renewcommand\t{\tau}
\newcommand\R{\mathbb R}\newcommand\N{\mathbb N}
\newcommand\C{\mathbb C}

\newcommand\dsp{\displaystyle}

\def\g{\gamma}

\def\t{\tau}

\def\l{\lambda}

\def\e{\varepsilon}

\newcommand\br{\begin{rem}}
\newcommand\er{\end{rem}}
\newcommand\bp{\begin{pmatrix}}
\newcommand\ep{\end{pmatrix}}
\newcommand\be{\begin{equation}}
\newcommand\ee{\end{equation}}
\newcommand\ba{\begin{equation}\begin{aligned}}
\newcommand\ea{\end{aligned}\end{equation}}

\newcommand{\Id}{{\rm Id }}

\newtheorem{defi}{Definition}[section]
\newtheorem{theo}[defi]{Theorem}
\newtheorem{prop}[defi]{Proposition}
\newtheorem{lem}[defi]{Lemma}
\newtheorem{cor}[defi]{Corollary}
\newtheorem{rem}[defi]{Remark}

\def\op{{\rm op} }

\numberwithin{equation}{section}
\begin{document}

\renewcommand{\refname}{References}

\title[Approximations of pseudo-differential flows]{Approximations of pseudo-differential flows}
\author{Benjamin Texier}

\address{Institut de Math\'ematiques de Jussieu-Paris Rive Gauche UMR CNRS 7586, Universit\'e Paris-Diderot}
\email{benjamin.texier@imj-prg.fr}

\thanks{${}^{}$ \quad The author acknowledges support from the Project ``Instabilities in Hydrodynamics'' funded by the Mairie de Paris (under the ``Emergence'' program) and the Fondation Sciences Math\'ematiques de Paris, thanks Nicolas Lerner and Kevin Zumbrun for interesting discussions, Hugo Federico and Hui Zhu for their comments on an earlier version of the manuscript, St\'ephane Nonnenmacher for pointing out reference \cite{DS}, and the anonymous referee for valuable comments.}

\date{\today}

\begin{abstract} Given a classical symbol $M$ of order zero, and associated semiclassical operators $\op_\e(M),$ we prove that the flow of $\op_\e(M)$ is well approximated, in time $O(|\ln \e|),$ by a pseudo-differential operator, the symbol of which is the flow $\exp(t M)$ of the symbol $M.$ A similar result holds for non-autonomous equations, associated with time-dependent families of symbols $M(t).$ This result was already used, by  the author and co-authors, to give a stability criterion for high-frequency WKB approximations, and to prove a strong Lax-Mizohata theorem. We give here two further applications: sharp semigroup bounds, implying nonlinear instability under the assumption of spectral instability at the symbolic level, and a new proof of sharp G\r{a}rding inequalities.
\end{abstract}

\maketitle

\section{Introduction}

Consider a family $\op_\e(M)$ of semiclassical pseudo-differential operators associated with a matrix-valued classical symbol $M$ of order zero: that is $M(x,\xi) \in \C^{n \times n },$ for $(x,\xi) \in \R^d \times \R^d,$ satisfying the uniform bounds
\be \label{0} |\d_x^\a \d_\xi^\b M(x,\xi)| \leq C_{\a\b} (1 + |\xi|^2)^{-|\b|/2}, \qquad C_{\a\b} > 0, \,\, \a, \b \in \N^d,\ee
and the associated family of operators defined on the Schwartz class by 
\be \label{def:op}
 (\op_\e(M) u)(x) = \int_{\R^d} e^{i x \cdot \xi} M(x, \e \xi) \hat u(\xi) \, d\xi, \qquad \e > 0.
 \ee
  By the Calder\'on-Vaillancourt theorem, for all $\e > 0,$ $\op_\e(M)$ extends to a linear bounded operator $L^2(\R^d) \to L^2(\R^d).$ Denote $\exp(t \op_\e(M))$ the flow of the ordinary differential equation
\be \label{1} 
 \d_t u = \op_\e(M) u,
\ee
which is known to exist and be global by the Cauchy-Lipschitz theorem, so that $\exp(t \op_\e(M)) u_0$ denotes the unique solution to \eqref{1} with value $u_0 \in L^2(\R^d)$ at $t = 0.$ 

\medskip

We show here that, in time $O(|\ln \e|),$ there holds the approximation
$$ %
 \exp(t \op_\e(M)) \simeq \op_\e(\exp( t M)),
$$ %
made precise in the Approximation Lemma \ref{lem:tildeS} below.

\medskip

More generally, given a bounded family $(M(t))_{t \in \R}$ in the space of symbols of order zero, we show that the solution to the initial value-problem
\be \label{3} 
 \d_t u  = \op_\e(M(t)) u, \qquad u(0)  = u_0 \in L^2(\R^d),\ee
is well approximated, in time $O(|\ln \e|),$ by $\op_\e(S(0;t)) u_0,$ where $S$ is the solution operator for $M(t),$ defined by 
\be \label{4}
\d_t S(\t;t) = M(t) S(\t;t), \qquad S(\t;\t)  \equiv \mbox{Id}.
\ee

 In other words, we approximate solution operators to a class of ordinary differential equations in infinite dimensions (typically, $L^2$) by pseudo-differential operators, the symbols of which are solution operators to ordinary differential equations in {\it finite} dimensions (typically, $\C^{n \times n }$).
 
 This reduction to finite dimensions has applications in particular to {\it stability} problems. Indeed, spectra of variable-coefficient (pseudo)-differential operators are typically difficult to describe, while the spectra of their symbols, being spectra of families of matrices, are at least theoretically computable. Indeed, the Approximation Lemma was already used by the author and co-authors:

\smallskip

$\bullet$\, in \cite{em4}, we proved that for large-amplitude high-frequency WKB solutions to semilinear hyperbolic systems, stability is generically equivalent to preservation of hyperbolicity around resonant frequencies. The verification of this stability criterion involves only computation of spectra and eigenprojectors in finite dimensions. This result applies in particular to instabilities in coupled Klein-Gordon systems and to the Raman and Brillouin instabilities. 

\smallskip

$\bullet$\, In \cite{LNT}, we proved a strong Lax-Mizohata theorem stating that even a weak defect of hyperbolicity implies ill-posedness for systems of first-order partial differential equations, extending work of M\'etivier \cite{Me}. 

\medskip

We give here two further applications:

\smallskip

$\bullet$\, in Theorem \ref{th:spectral}, Section \ref{sec:ex}, sharp lower and upper bounds are proved for the solution operator to \eqref{1}; in line with the above comment following equation \eqref{4}, we note that we dispense here with any consideration of infinite-dimensional spectra of linear (pseudo)-differential operator, and derive growth estimates based solely on consideration of spectra of matrices (symbols). In Section \ref{sec:spmap}, we observe that the bounds of Theorem \ref{th:spectral} are typically sharper than bounds derived from G\r{a}rding's inequality, and in Section \ref{sec:insta} we use Theorem \ref{th:spectral} to prove a nonlinear instability result for an ordinary differential equation in Sobolev spaces. 

\smallskip

$\bullet$\, In Section \ref{sec:Fef-pho}, we give a new proof of sharp G\r{a}rding inequalities with gain of $\theta$ derivatives, for $0 < \theta < 1,$ based on the Approximation Lemma \ref{lem:tildeS}. This somehow completes the comparison, initiated in Section \ref{sec:ex}, of Lemma \ref{lem:tildeS} with G\r{a}rding's inequality.

\medskip

We conclude this introduction with three remarks:

\smallskip

$\bullet$\, {\it the assumption that $M$ be order zero is crucial for our purposes.} Indeed, for the exponential $e^M$ of a classical symbol $M$ to be itself a symbol, we typically need $M \in S^0.$  We can, however, do without the semiclassical quantization in \eqref{1} and \eqref{3}. Indeed, in Section \ref{sec:Fef-pho}, we prove an Approximation Lemma for symbols in Weyl quantization; powers of $\e$ are there replaced with gains in the orders of the operators.

\smallskip

$\bullet$\, {\it There can be found in the literature a number of results describing solution operators to pseudo-differential equations in terms of pseudo-differential operators;} for instance Th\'eor\`eme 6.4 in \cite{BC} and Lemma 8.5 in \cite{Z}, both based on Beals's lemma characterizing pseudo-differential operators (Proposition 8.3 in \cite{DS}). Thus the novelty here is not the description of solution operators as pseudo-differential operators (although we could not find in the literature statements equivalent to Lemma \ref{lem:tildeS} and Theorem \ref{th:duh}), but rather the use we make of this description, in Sections \ref{sec:ex} and \ref{sec:Fef-pho}, in the case of real symbols, or symbols with spectra that are not purely imaginary.   

\smallskip

$\bullet$\, {\it The time $O(|\ln \e|)$ is small compared to standard observation times in the semiclassical limit.} The semiclassical limit is concerned with operators $\e \d_s - A_\e,$ where $\e$ is the semiclassical parameter, and, for instance, $A_\e = i (\e^2 \Delta - V),$ for some potential $V(x)$ (see for instance \cite{BR,Z}). Rescaling the time, $t = s/\e,$ and applying a frequency truncation operator $\chi(\e D),$ with $\chi \in C^\infty_c(\R^d_\xi),$ we find operator $\d_t - \chi(\e D) A_\e,$ to which the Approximation Lemma \ref{lem:tildeS} applies, without any assumption on the potential $V,$ but only in time $O(|\ln \e|)$ in the fast variable $t,$ corresponding to small time $O(\e |\ln \e|)$ in the original temporal variable.
 \section{The approximation Lemma} \label{sec:Duhamel}

 Let $M(t)$ be a bounded family in $S^0,$ meaning a family of smooth maps $(t,x,\xi) \in \R \times \R^d \times \R^d \to M(t,x,\xi) \in \C^{n \times n },$ such that the bounds \eqref{0} hold uniformly in $(t,x,\xi).$ Consider the associated ordinary differential equations
 \begin{equation} \label{buff0} \left\{
  \begin{aligned} \d_t u & = \op_\e(M) u + f, \\ u(0) & = u_0.\end{aligned}\right.  
  \end{equation}
 where $\op_\e(M)$ is defined in \eqref{def:op}. In \eqref{buff0}, the datum $u_0$ belongs to $H^s,$ and the source $f$ is given in $C^0([0, T |\ln\e|], H^s(\R^d)),$ for some $s \in \R,$ some $T > 0.$ 

Let $S(\t;t)$ be the (finite-dimensional) solution operator associated with $M(t),$ that is the family of solutions to the ordinary differential equations in $\C^{n \times n }:$ 
 \be \label{def:S} 
 \d_t S(\t;t) = M(t) S(\t;t), \qquad S(\t;\t) \equiv \Id.
 \ee
 By how much does $\op_\e(S(0;t))$ fail to be the operator solution to \eqref{buff0}? By composition of operators in semiclassical quantization, there holds
 \be \label{error0} \op_\e(\d_t S) = \op_\e(M S) = \op_\e(M) \op_\e(S) - \e \underbrace{\op_\e(M \sharp S) - \e^2 (\dots)}_{\footnotesize{\mbox{error term}}},\ee
 where $\sharp$ denotes the bilinear map $\dsp{\s_1 \sharp \s_2 := \sum_{|\a| = 1} - i \d_\xi^\a \s_1 \d_x^\a \s_2.}$
 Classical results on pseudo-differential operators are recalled in the Appendix (Section \ref{sec:symbols}); in particular a precise estimate for the error in \eqref{error0} is given in \eqref{compo:e}-\eqref{composition:e}.
 
 We see in \eqref{error0} that the leading term in the error is presumably $\e \op_\e(M \sharp S),$ which, in times $O(|\ln \e|),$ may be catastrophically large. 
 
 Indeed, there holds, by Gronwall's lemma, the bound 
 \be \label{basic:S} |S(\t;t)| \leq e^{\g (t - \t)}, \quad \mbox{where $\g = | M |_{L^\infty(t,x,\xi)}.$}\ee
 In the autonomous case $M = M(x,\xi),$ then $S(\t;t) = \exp((t - \t) M(x,\xi)),$ and there holds the more precise bound 
 \be \label{basic:S:aut} |S(\t;t)| \leq C(|M|_{L^\infty}) (1 + t)^{n-1} e^{\g (t - \t)}, \quad \mbox{where $\g = \sup_{x,\xi} \Re e \, \s(M(x,\xi)),$}\ee
 for some $C(|M|_{L^\infty}) > 0,$ where $\s(M(x,\xi))$ denotes the symbol of the matrix $M(x,\xi).$ 
 
 In the following we denote $\g$ both growth rates in \eqref{basic:S} and \eqref{basic:S:aut}; if $M$ is not specified to be time-independent, then \eqref{basic:S} applies. 

  This implies, via the representation
$$ %
 \d_x^\a \d_\xi^\b  S(\t;t) = \int_\t^t S(t';t) [ \d_x^\a \d_\xi^\b , M(t')] S(\t;t') \, dt',
$$ %
where $[A,B] = A B - B A$ (commutator), the bounds 
 \be \label{bd:S0}
  \langle \xi \rangle^{|\b|} |\d_{x}^\a \d_\xi^\b S(\t;t)| \lesssim |\ln \e|^{*} e^{\g (t - \t)}, \quad \mbox{for $\t \leq t,$}
  \ee
 where $|\ln \e|^* = |\ln \e|^{N^*},$ for some $N^* = N^*(\a,\b,n) \in \N,$ and $\lesssim$ means inequality up to a multiplicative constant, depending on $\a,$ $\b,$ $M$ and $T$ but {\it not} on $(\e,\t,t).$ Thus there holds $\e |M \sharp S| \lesssim \e |\ln \e|^* e^{\g t},$ and the upper bound is very large in time $O(|\ln \e|),$ in spite of the $\e$ prefactor. %

 We then introduce a first-order corrector $S_1,$ defined by 
 $$ \d_t S_1 = M S_1 + M \sharp S, \qquad S_1(\t;\t) = 0,$$
 so that
$$ %
 S_1(\t;t) = \int_\t^t S(s;t) M(s) \sharp S(\t;s) \, ds.
 $$ %
In particular, $S_1 \in S^{-1},$ and in time $O(|\ln \e|)$ the corrector $S_1$ and its derivatives are growing at most at exponential rate $\g,$ no faster than $S,$ up to a prefactor of the form $|\ln \e|^*,$ precisely:
\be \label{bd:S1}
  \langle \xi \rangle^{1 + |\b|} |\d_{x}^\a \d_\xi^\b  S_1| \lesssim |\ln \e|^{*} e^{\g (t - \t)}.
  \ee
  The symbol $S_0 + \e S_1$ is a candidate for a better approximation of the symbol of the solution operator, in that it satisfies
 \be \label{eq:s1} \d_t \op_\e(S_0 + \e S_1) = \op_\e(M) \op_\e(S_0 + \e S_1) - \e^2 (\dots).\ee
In the above error $O(\e^2),$ the leading term involves symbols like $M \sharp S_1,$ which is not growing faster than $S.$ Thus the error in \eqref{eq:s1} is truly smaller than the error in \eqref{error0}: the net gain is a power of $\e,$ modulo possibly large, and essentially irrelevant, powers of $|\ln \e|.$

  Iterating this procedure, we define $(S_q)_{1 \leq q \leq q_0},$ for $q_0 := [\g T] + 1,$ as the solution to the triangular system of linear ordinary differential equations
 \begin{equation} \label{def-Sk}
  \d_t S_q = M S_q + \sum_{\begin{smallmatrix} q_1 + q_2 = q \\ 0 <  q_1 \end{smallmatrix}} M \sharp_{q_1} S_{q_2},  \quad S_q(\t;\t) = 0, \quad 1 \leq q \leq q_0 = [\g T] + 1,
  \end{equation} 
where the bilinear map $\sharp_q$ is defined by 
$$ %
 a_1 \sharp_q a_2 := \sum_{|\a| = q} \frac{(-i)^{|\a|}}{|\a|!} \d_\xi^\a a_1 \d_x^\a a_2.
$$ %
From \eqref{def-Sk}, we see that $S_q$ satisfies bounds 
\be \label{bd:Sq} \langle \xi \rangle^{q + |\b|} |\d_x^\a \d_\xi^\b S_q(\t;t)| \lesssim |\ln \e|^* e^{\g (t - \t)}.
\ee The approximate solution operator is defined as
 \begin{equation} \label{def-S}
  \Sigma = S + \sum_{1 \leq q \leq q_0} \e^q S_q.
   \end{equation} 

 \begin{lem}[Approximation Lemma] \label{lem:tildeS} Given $T > 0,$ given a bounded family $M = M(t)$ in $S^0,$ as defined in the first lines of this Section, the operator $\op_\e(\Sigma(\t;t))$ defined in \eqref{def-S} is an approximate solution operator over $[0, T |\ln \e|]$ for the differential equation \eqref{buff0}, in that it satisfies%
  \begin{equation} \label{tildeS} \d_t \op_\e(\Sigma) = \op_\e(M) \op_\e(\Sigma) + \e \op_\e(\rho),\end{equation}
  with $\rho$ such that, for all $u \in H^{s-q_0 - 1}(\R^d),$ 
  \begin{equation} \label{tilde-remainder} \|  \op_\e(\rho(\t;t)) u \|_{\e,s} \lesssim |\ln \e|^{*} \| u\|_{\e, s-q_0 - 1},\end{equation}
  uniformly in $0 \leq \t \leq t \leq T |\ln \e|.$ 
  \end{lem}
  
  Above, $\|\cdot\|_{\e,s}$ denotes the semiclassical Sobolev norm $\| u \|_{\e,s} := \big| (1 + |\e \xi|^2)^{s/2} \hat u|_{L^2(\R^d_\xi)}.$ 
  
  \smallskip
  
  The index $s \in \R$ is arbitrary, equal to the assumed regularity of both the source $f$ in \eqref{buff0} and the datum $u_0.$ 

\smallskip
  
  The Taylor index $q_0$ is defined by $q_0 = [\g T] + 1,$ where $\g = |M|_{L^\infty(t,x,\xi)}.$ As noted in \eqref{basic:S:aut}, in the autonomous case we may use $\g = \sup_{x,\xi} \Re e \, \s(M(x,\xi)).$  

\smallskip
  
   Also, in \eqref{tilde-remainder} the symbol $\lesssim$ denotes upper bound up to a multiplicative constant (which does not depend on $\e, \t, t$), and $|\ln \e|^*$ means $|\ln \e|^{N^*},$ for some $N^* > 0$ possibly dependent on all parameters, in particular on dimensions $d, n,$ on the symbol $M,$ on $q_0,$ but not on $\e, \t, t.$   %

 \begin{proof} By composition of operators (see \eqref{compo:e}-\eqref{composition:e}-\eqref{continuite:s}), there holds for $q \geq 0,$ denoting $S_0 : =S,$  
 \be \label{compo:crux} \e^q \op_\e(M) \op_\e(S_q) = \e^q \op_\e(M S_q) + \sum_{ 1 \leq q_1 \leq q_0 - q} \e^{q + q_1} \op_\e( M\sharp_{q_1}   S_q ) + \e^{q_0 + 1} \op_\e(\rho_q),\ee
 where $\rho_q = R_{q_0 - q + 1}(M, S_q),$ using notation introduced in \eqref{compo:e},  satisfies the bound
 \be \label{bd:rhoq} \e^{q_0 + 1} \| \op_\e(\rho_q(\t;t)) u \|_{\e,s} \lesssim \e^{q_0 + 1} |\ln \e|^* e^{\g (t - \t)} \| u \|_{\e,s-q_0 -1}, \ee %
  for all $u \in H^{-q_0 -1},$ uniformly in $0 \leq \t \leq t \leq T |\ln \e|.$ %
  Let $\rho := - \sum_{0 \leq q \leq q_0} \rho_q.$ Summing \eqref{compo:crux} over $q,$ we obtain 
 \be \label{forApp} \op_\e(M) \op_\e(\Sigma) = \op_\e(M \Sigma) + \sum_{\begin{smallmatrix} 0 \leq q \leq q_0 \\ 1 \leq q_1 \leq q_0 - q \end{smallmatrix}}  \e^{q_1} \op_\e( M \sharp_{q_1} S_q ) - \e \op_\e(\rho).\ee
 Besides, by definition of the correctors \eqref{def-Sk}, there holds
 $$ \d_t \op_\e(\Sigma) = \op_\e(M \Sigma) + \sum_{\begin{smallmatrix} 0 \leq q_1 + q_2 \leq q_0 \\ 0 < q_1 \end{smallmatrix}} \e^{q_1 + q_2} \op_\e(M \sharp_{q_1} S_{q_2}),$$
 and comparing with \eqref{forApp} we obtain identity \eqref{tildeS}. The remainder $\rho$ satisfies \eqref{tilde-remainder}, simply by summation of bounds \eqref{bd:rhoq}, since by choice of $q_0$ there holds $\e^{q_0 + 1} e^{\g (t - \tau)} \leq \e.$
  \end{proof}

 The Approximation Lemma \ref{lem:tildeS} leads to a representation theorem for \eqref{buff0}:

 \begin{theo} \label{th:duh} Given $T > 0,$ given a bounded family $M = M(t)$ in $S^0,$ for $\e$ small enough, the unique solution $u \in C^0([0,T |\ln \e|], H^s(\R^d))$ to the initial-value problem \eqref{buff0} is given by 
  \begin{equation} \label{buff0.0} u = \op_\e(\Sigma(0;t)) u_0 + \int_0^t \op_\e(\Sigma(t';t))(\Id + \e R)\Big( f - \e \op_\e(\rho(0;\cdot)) u_0\Big)(t') \, dt',
  \end{equation}
  where $\Sigma$ is defined in \eqref{def-S}, $\rho$ is the remainder in the Approximation Lemma {\rm \ref{lem:tildeS}}, and $R $ is linear bounded $C^0([0,T |\ln \e|], H^{s-q_0-1}(\R^d)) \to C^0([0,T |\ln \e|, H^s(\R^d)),$ with bound   
  \begin{equation} \label{bd-R12} \begin{aligned} 
   \sup_{0 \leq t \leq T |\ln \e|} \| (R w)(t) \|_{\e,s} & \lesssim |\ln \e|^{*} \sup_{0 \leq t \leq T |\ln \e|} \| w(t) \|_{\e,s-q_0 -1}, \end{aligned}
  \end{equation}
 for all $w \in C^0([0, T |\ln \e|], H^{s-q_0-1}(\R^d)).$%
 \end{theo}

In \eqref{bd-R12}, we use the convention for $\lesssim$ and $|\ln \e|^*$ introduced just below Lemma \ref{lem:tildeS}.

\begin{proof} Let $g \in C^0([0,T |\ln \e|], H^s(\R^d)).$ By Lemma \ref{lem:tildeS}, the map $u$ defined by 
 $$ %
 u := \op_\e( \Sigma(0;t)) u_0 + \int_0^t \op_\e( \Sigma(t';t)) g(t') \, dt'
 $$ %
   solves \eqref{buff0} if and only if, for all $t,$ 
\begin{equation} \label{cond-g}
  \big((\Id + \rho_0) g\big)(t)  = f(t) - \e \op_\e(\rho(0;t)) u_0,
   \end{equation}
 where $\rho$ is the remainder introduced in Lemma \ref{lem:tildeS} and $\rho_0$ is the linear integral operator 
  $$\rho_0: \qquad v \in C^0([0, T |\ln \e|], H^s) \to \Big(t \to \e \int_0^t \op_\e(\rho(t';t)) v(t') \, dt'\Big) \in C^0([0, T |\ln \e|], H^s).$$
   By \eqref{tilde-remainder}, there holds the uniform bound 
  \begin{equation} \label{bd-rf} \sup_{0 \leq t \leq T |\ln \e|} \| (\rho_0 u)(t) \|_{\e,s} \leq \e C|\ln \e|^{*} \sup_{0 \leq t \leq T |\ln \e|} \| u(t) \|_{\e,s-q_0 -1},
  \end{equation}
  for some $C = C(d,n,M,T) > 0.$ 
This implies that, for $\e$ small enough, depending on $d, n, M, T,$ the operator $\Id + \rho_0$ is invertible in the Banach algebra ${\mathcal L}(C^0([0, T |\ln \e|], H^s)).$ As a consequence, we can solve \eqref{cond-g} in $C^0([0, T |\ln \e|], H^s),$ 
 and obtain the representation formula \eqref{buff0.0}, in which
  the remainder $\e R :=  (\Id + \rho_0)^{-1} - \Id$
   satisfies \eqref{bd-R12}, by \eqref{bd-rf}.  
\end{proof}

It can be useful to know exactly how small $\e$ is required to be in Theorem {\rm \ref{th:duh}}. %

\begin{prop} \label{rem:eps} Given $T > 0,$ given a bounded family $M = M(t)$ in $S^0,$ given $f \in L^\infty([0, T |\ln \e|, H^s(\R^d)),$ given $u_0 \in H^s(\R^d),$ Theorem {\rm \ref{th:duh}} applies as soon as 
\be \label{cond:eps}
   \e |\ln \e|^{N^*} < C_0 \big( \sup_{0 \leq t \leq T |\ln \e|} \| M(t) \|_{\g T + C(d)} \big)^{-1}\,,
 \ee
  with constants $C_0 = C_0(d,n) > 0,$ where $d$ is the spatial dimension and $n$ the size of matrix $M(t,x,\xi),$ and $C(d) > 0,$ $N^* = N^*(\g T, n, d) > 0,$ with $\g$ given by \eqref{basic:S} or \eqref{basic:S:aut}.
\end{prop}

In particular, the order of regularity required for the symbol, namely $\g T + C(d),$  is a function of its $L^\infty$ norm $\g$.

Above, the notation $\| \cdot \|_r$ denotes the symbolic norm introduced in \eqref{symb:norms}.

\begin{proof}
A look at the proofs of Lemma {\rm \ref{lem:tildeS}} and Theorem {\rm \ref{th:duh}} shows that we need $0 < \e < 1$ to be small enough so that $\Id + \rho_0$ be invertible. According to bound \eqref{bd-rf}, this is implied by $C \e |\ln \e|^{*} < 1,$ for some constant $C$ depending in particular on $M$ and $\g T.$ How specifically this constant $C$ depends on $M$ and $T$ is found by going back to bound \eqref{bd:Sq} for the correctors, from which \eqref{bd-rf} derives. In \eqref{bd:Sq}, there appears an implicit multiplicative constant, which depends on $M$ through $\| M \|_{q_0 + C(d)},$ for some $C(d)$ depending only on $d.$ Also, in \eqref{bd:Sq} the implicit exponent in $|\ln \e|^*$ depends on $q_0$ and $n.$ Thus condition $C \e |\ln \e|^* < 1$ takes the form \eqref{cond:eps}. 
\end{proof}

\section{Application: sharp semigroup bounds} \label{sec:ex}

The results of Section \ref{sec:Duhamel} translate into sharp semigroup bounds. Here we restrict ourselves to the case of autonomous symbols $M \in S^0,$ and denote $\s(M(x,\xi))$ the spectrum of matrix $M(x,\xi) \in \C^{n \times n},$ for $(x,\xi) \in \R^d \times \R^d.$  

\begin{theo} \label{th:spectral} Given $T > 0$ and $M \in S^0,$ there holds for $\e$ small enough the upper bound 
 \be \label{up} \big| \exp \big( t \op_\e(M) \big) \big|_{L^2(\R^d) \to L^2(\R^d)} \lesssim |\ln \e|^{*} \exp \big( t \sup_{x,\xi} \Re e \, \sigma(M(x,\xi))\big),\ee
 uniformly in $t \in [0, T |\ln \e|],$ and, for some $u_\e$ on the unit sphere of $L^2(\R^d),$ some $x_0 \in \R^d,$ the lower bound
  \be \label{low}  \exp \big( t \sup_{x,\xi} \Re e \, \sigma(M(x,\xi)) \big) \lesssim |\ln \e|^* \big| \exp \big( t \op_\e(M) \big) u_{\e} \big|_{L^2(B(x_0, |\ln \e|^{-1})},\ee
 uniformly in $t \in [0, T |\ln \e|].$ 
\end{theo}

We use in the statement of Theorem \ref{th:spectral} the convention for $\lesssim$ and $|\ln \e|^*$ introduced below Lemma \ref{lem:tildeS}.

The bounds of Theorem \ref{th:spectral} are sharp in the sense that the lower growth rate is equal to the upper growth rate. Here we mean {\it exponential rate of temporal growth}, equal to $\sup_{x,\xi} \Re e \, \sigma(M).$ 

In the case of a symbol $M = M(x)$ that is independent of $\xi,$ elementary linear-algebraic bounds give this sharp rate of growth.

 The proof below, and Proposition \ref{rem:eps} above, show that in order for \eqref{up} and \eqref{low} to hold, we need $\e$ to satisfy a bound of the form $\e |\ln \e|^{N^*(M,T)} < C(M,T).$

\medskip

\begin{proof} By Lemma \ref{lem:tildeS} and Theorem \ref{th:duh}, there holds the bound, for $\e$ small enough,
 \be \label{3.0} \big| \exp (t \op_\e(M))|_{L^2 \to L^2} \leq |\op_\e(\Sigma(0;t))|_{L^2 \to L^2} + \e |\ln \e|^* \int_0^t |\op_\e(\Sigma(t';t))|_{L^2 \to L^2} \, dt',\ee
 where $\Sigma$ is the approximate solution operator defined in \eqref{def-S}. As already mentioned in Section \ref{sec:Duhamel}, since $M$ is autonomous, there holds $S(\t;t,x,\xi) = \exp((t - \t) M(x,\xi)).$ Thus bound \eqref{basic:S:aut} holds, and in the bounds \eqref{bd:S0} and \eqref{bd:S1} for $S$ and $S_1,$ and similarly in the bounds for the higher-order correctors $S_q,$ we may use $\g = \sup_{x,\xi} \Re e \, \s(M).$ Bounds \eqref{bd:S0} and \eqref{bd:S1} provide the bounds for the symbols that are used in the Calder\'on-Vaillancourt theorem (bound \eqref{action} in Section \ref{sec:symbols}), and this implies for $0 \leq t' \leq t \leq T |\ln \e|$ the bound 
 \be \label{3.00} |\op_\e(\Sigma(t';t))|_{L^2 \to L^2} \lesssim |\ln \e|^{*} e^{(t - t') \g}.\ee
 Using \eqref{3.00} in \eqref{3.0} we find \eqref{up}. 

\medskip
 
 We turn to a proof of the lower bound \eqref{low}. For any $\delta > 0,$ which we will eventually choose appropriately small, depending on $\e,$ there can be found $(x_0,\xi_0) \in \R^d \times \R^d$ and $\l_0$ in the spectrum of $M(x_0,\xi_0)$ such that $\Re e \, \l_0 - \delta > \g = \sup_{x,\xi} \Re e \, \s(M(x,\xi)).$ Let $u_\e(x) := e^{i x \cdot \xi_0/\e} \theta(x),$ where $\xi_0 \in \R^d$ and $\theta \in C^\infty_c(\R^d)$ is compactly supported around $x_0,$ with $|\theta|_{L^2} = 1.$ There holds  
 \be \label{op:dec:0} \op_\e(e^{t M}) u_\e = e^{i x \cdot \xi_0/\e} \int_{\R^d} e^{i x \cdot \xi} e^{t M(x,\xi_0 + \e \xi)} \hat \theta(\xi) \, d\xi = e^{i x \cdot \xi_0/\e} e^{t M(x, \xi_0)} \theta(x) + \e v_\e,\ee
 where
 $$ v_\e := e^{i x \cdot \xi_0/\e} \int_0^1 \int_{\R^d} e^{i x \cdot \xi} (\d_\xi e^{t M})(x,\xi_0 + \e \t \xi)  \widehat{\d_x \theta}(\xi) \, d\xi \, d\t. $$ %
Thus the first task ahead is to find a bound from below for the family of vectors $e^{t M(x,\xi_0)} \theta(x).$ Here we follow closely the proof of Theorem 1.2 in \cite{LNT}.  

\medskip

{\it First step.} Let $Q_0$ be the sum of the projectors onto the generalized eigenspaces associated with the eigenvalues $\l$ of $M(\cdot,\xi_0)$ which satisfy $\l \to \l_0$ as $x \to x_0,$ and parallel to the sum of the other generalized eigenspaces of $M(\cdot,\xi_0).$ If $\l_0$ is a coalescing point in the spectrum of $M(x_0,\xi_0)$ (that is, if there is more than one branch of eigenvalues $\l$ which satisfy $\l(x_0,\xi_0) = \l_0$), then the eigenvalues may not be smooth, and the corresponding eigenprojectors not even defined at $(x_0,\xi_0),$ but the eigenvalues are continuous (see for instance \cite{K}, or Proposition 1.1 in \cite{T}), and the projector $Q_0$ is smooth (see for instance Proposition 2.1 in \cite{T}). There holds $Q_0 \exp(t M) Q_0^{-1} = \exp( t Q_0 M Q_0^{-1}).$ We may choose $\theta(x) = \theta_0(x) Q_0(x)^{-1} \vec e,$ for a fixed vector $\vec e$ and a scalar $\theta_0.$ We let $M_0(x) = Q_0(x) M(x,\xi_0) Q_0(x)^{-1}.$ Then
 \be \label{eq:Q0} Q_0(x) \exp(t M(x,\xi_0)) \theta(x) = \theta_0(x) e^{t M_0(x)} \vec e.\ee
 The matrix $M_0$ is block diagonal. By choice of $Q_0,$ the eigenvalues of the top block converge to $\l_0$ as $x \to x_0,$ and the other block is identically zero. We may restrict our attention to the top left block, that is, assume that its size is $n.$ 

\medskip

{\it Second step.} Let $P$ be a constant invertible matrix so that $P M_0(x_0) P^{-1}$ is upper triangular, with $\l_0$ on the diagonal (by virtue of the above first step). Let $P_\mu$ be the diagonal matrix with diagonal entries $1, \mu^{-1}, \mu^{-2}, \dots, \mu^{1-n},$ for some $\mu > 0$ which we will eventually choose to be small, depending on $\e.$ There holds
$$ P_\mu P M_0(x_0) P^{-1} P_\mu^{-1} = \l_0 \Id + \mu J,$$
where $J$ is upper triangular, with zeros on the diagonal, and with entries over the diagonal which are all $O(\mu).$ By regularity of $M,$ 
$$  P_\mu P M_0(x) P^{-1} P_\mu^{-1} = \l_0 \Id + \mu J + (x - x_0) \cdot M_1(x),$$
for some bounded family of matrices $M_1.$ Further choosing the vector $\vec e$ from the first step in the form $\vec e = P^{-1} P_\mu^{-1} \tilde e,$ we thus consider  
\be \label{eq:Pmu} P_\mu P e^{t M_0(x)} \vec e = e^{t \l_0} e^{\mu J + (x - x_0) \cdot M_1(x)} \tilde e.\ee

\medskip

{\it Third step.} For any matrix $N,$ we denote $\Re e \, N$ the hermitian matrix $\Re e \, N = (N + N^*)/2.$ The eigenvalues $\tilde \l(x)$ of the Hermitian matrix $\Re e \, (\mu J + (x - x_0) \cdot M_1(x))$ are semisimple. In particular, they satisfy (see for instance Proposition 3.2 in \cite{T}):
 $$ |\tilde \l(x)| \lesssim \mu + |x - x_0| =: \tilde \g(x),$$
so that 
\be \label{bound:hermitian}
 \Re e \, (\mu J + (x - x_0) \cdot M_1(x)) + \tilde \g(x) \geq 0,  
\ee
in the sense of hermitian matrices. 
 Let $\tilde S(t,x) = e^{t(\mu J + (x - x_0) M_1(x))}.$ There holds
 $$ \frac{1}{2} \d_t (|\tilde S(t,x) \tilde e|^2) = \Big( \, \Re e\, (\mu J + (x - x_0) \cdot M_1(x))  \tilde S(t,x)\tilde e, \, \tilde S(t,x)\tilde e \, \Big),$$
 where $(\cdot,\cdot)$ denotes the hermitian scalar product in $\C^n.$ By \eqref{bound:hermitian}, this implies 
$$ %
 |\tilde S(t,x) \tilde e\,| \geq e^{-2 t \tilde \g(x)} |\tilde e|.
$$ %
Thus with \eqref{eq:Pmu},
$$ |P_\mu P e^{t M_0(x)} \vec e| \geq e^{t (\Re e \, \l_0 - 2 \tilde \g(x))} |\tilde e|,$$
and, with \eqref{eq:Q0}, we arrive at
$$ |e^{t M(x,\xi_0)} \theta(x)| \geq e^{t (\Re e \, \l_0 - 2 \tilde \g(x))}  |P_\mu P Q_0(x)|^{-1} |\theta_0(x)| |\tilde e|.$$
We let $\mu = |\ln \e|^{-1}.$ Then, for $|x - x_0|= O(|\ln \e|^{-1}),$ the exponential $e^{-2 t \tilde \g(x)}$ is bounded away from zero for $t = O(|\ln \e|).$ By regularity of $Q_0$ and choice of $\mu,$ there holds the lower bound $|P_\mu P Q_0(x)|^{-1} \gtrsim |\ln \e|^{-(n-1)}.$ We choose $\tilde e = (1,0,\dots,0).$ Then, $P_\mu^{-1} \tilde e = \tilde e,$ and 
$$ |e^{t M(\cdot,\xi_0)} \theta|_{L^2(B(x_0,|\ln\e|^{-1}))} \gtrsim |\ln \e|^{-(n-1)} e^{t \Re e \, \l_0} |\theta_0|_{L^2(B(x_0,|\ln\e|^{-1}))}.$$
Since $|\theta_0|_{L^2} = 1,$ with $\theta_0$ continuous and such that $\theta_0(x_0) = 1,$ the localized $L^2$ norm $|\theta_0|_{L^2(B(x_0,|\ln\e|^{-1}))}$ is bounded from below by a constant times $|\ln \e|^{-d}.$ Thus we obtain
$$ %
  e^{t \Re e \l_0} \lesssim |\ln \e|^* |e^{t M(\cdot,\xi_0)} \theta|_{L^2(B(x_0,|\ln\e|^{-1}))}.
$$ %
Further choosing $\delta = |\ln \e|^{-1},$ this gives
\be \label{final:3} 
 e^{t \g} \lesssim |\ln \e|^* |e^{t M(\cdot,\xi_0)} \theta|_{L^2(B(x_0,|\ln\e|^{-1}))}.
\ee

\medskip

{\it Conclusion.} We go back to \eqref{op:dec:0}.
We may use bound \eqref{action} in order to control the $L^2$ norm of $v_\e.$ Here we note that $\xi_0$ is a priori $\e$-dependent, but this does not affect the estimate for $v_\e,$ since the symbolic norms that appear in the Calder\'on-Vaillancourt bound are translation invariant. Thus
\be \label{bd:ve} |v_\e|_{L^2(\R^d)} \lesssim |\ln \e| e^{t \g} |\d_x \theta|_{L^2(\R^d)},\ee
the $|\ln\e|$ factor coming from the $t$ prefactor in  $\d_\xi (e^{t M}).$ By Theorem \ref{th:duh},
 $$ e^{t \op_\e(M)} u_\e = \op_\e(e^{t M}) u_\e - ({\rm I} + {\rm II}),$$
 where ${\rm I}$ is the contribution of the correctors $S_q$ to $\Sigma,$ and ${\rm II}$ is the time-integrated error term in \eqref{buff0.0}. By \eqref{bd:Sq}, \eqref{bd-R12} and \eqref{3.00}, both ${\rm I}$ and ${\rm II}$ are $O(\e |\ln \e|^* e^{t \g})$ in $L^2.$ Hence \eqref{op:dec:0}, \eqref{final:3} and \eqref{bd:ve} yield the result. \end{proof}

The following upper-growth bound in Sobolev and sup norms will be useful in our application in Section \ref{sec:insta}. 

\begin{theo} \label{prop:s} Given $T > 0,$ $M \in S^0,$ there holds for $\e$ small enough: for any $s \in \R,$ any $u \in H^s(\R^d)$ 
 \be \label{up:s}
  \big\| \exp \big( t \op_\e(M) \big) u \big\|_{\e,s}  \lesssim |\ln \e|^{*} \exp \big( t \sup_{x,\xi} \Re e \, \sigma(M(x,\xi))\big) \| u \|_{\e,s},\ee
 uniformly in $t \in [0, T |\ln \e|],$ and, for any $s > d/2,$ any $u \in H^{s}(\R^d),$ 
 \be \label{up:infty} \begin{aligned} \big| \exp \big( t \op_\e(M) \big) u \big|_{L^\infty} &  \lesssim |\ln \e|^{*} \exp \big( t \sup_{x,\xi} \Re e \, \sigma(M(x,\xi))\big) | u |_{L^\infty} \Big(1 + \ln\Big(\frac{\|u\|_{\e,s}}{|u|_{L^\infty}}\Big)\Big), \end{aligned} \ee
 uniformly in $t \in [0, T |\ln \e|].$
\end{theo} 

The semiclassical Sobolev norms $\| \cdot \|_{\e,s}$ are defined in \eqref{semicl}.

\smallskip

The difference between \eqref{up} and \eqref{up:s} is subsumed in the semi-implicit factor $|\ln \e|^*.$ 

\begin{proof} We argue as in the first part of the proof of Theorem \ref{th:spectral}, and use the Sobolev bound \eqref{continuite:s} for \eqref{up:s} and the pointwise bound \eqref{sup:pw} for \eqref{up:infty}. 
\end{proof}

\subsection{Comparison with the spectral mapping theorem and G\r{a}rding's inequality} \label{sec:spmap}

We argue here that the bounds of Theorem \ref{th:spectral} are more useful than bounds derived from the spectral mapping theorem, and sharper than bounds derived from G\r{a}rding's inequality.

\begin{prop}[Upper bound via the spectral mapping theorem]  \label{prop:spectral:mapping} Given $M \in S^0,$ there holds  
$$ %
 \big| \exp\big( t \op_\e(M)\big) \big) \big|_{L^2(\R^d) \to L^2(\R^d)} \lesssim C(\e,\delta) \exp\big( t (\delta + \g_\e) \big),
$$ %
 for any $\delta > 0,$ all $t \geq 0,$ some $C(\e,\delta ) > 0,$ with the growth rate 
 $$\g_\e := \sup \Re e \, \sigma(\op_\e(M)).$$
\end{prop}
 
  The proof below shows that the constant $C(\e,\delta)$ is bounded as $\e \to 0$ and typically unbounded as $\delta \to 0.$ 
  
 \begin{proof} First, by the semigroup property, there holds
$$ %
  \big| \exp\big( t \op_\e(M)\big) \big) \big|_{L^2 \to L^2} \leq c_0 \big| \big(\exp\big(\op_\e(M)\big)\big)^{[t]} \big|_{L^2 \to L^2},$$ %
 where 
 $c_0(\e) = \max_{0 \leq s \leq 1} |\exp(s \op_\e(M))|_{L^2 \to L^2}.$ Second, by Gelfand's formula, for any $\eta >0,$ for $t \geq t(\eta),$ there holds  
$$ %
 \big| \big(\exp\big(\op_\e(M)\big)\big)^{[t]} \big|_{L^2 \to L^2} \leq \Big(\rho\big(\exp\op_\e(M) \big) \big) + \eta\Big)^{[t]}.
 $$ %
 Third, as a consequence of the spectral mapping theorem (see for instance \cite{EG}, Lemma 3.13), 
$$ %
 \rho(\exp\op_\e(M)) \leq \exp\big( \sup \Re e \, \s(\op_\e(M))\big).$$ %
 Finally, given $\delta > 0,$ for some $\eta = \eta(\delta),$ the above bounds combine to yield the result. 
\end{proof}

In view of Proposition \ref{prop:spectral:mapping}, we may argue that Theorem {\rm \ref{th:spectral}} corresponds to a reduction to finite dimensions, as announced in the introduction. Indeed, the growth rate $\g_\e$ involves the spectra of the $L^2 \to L^2$ operators $\op_\e(M),$ while the upper bound \eqref{up} in Theorem \ref{th:spectral} involves only the $n \times n $ matrices $M(x,\xi).$ In particular, $\g_\e$ might be very difficult to compute, while $\sup \Re e \, \s(M)$ is readily computable, at least in theory. 
 
 \medskip

 Another classical way to derive semigroup bounds is G\r{a}rding's inequality.%

\begin{prop}[Upper bound via G\r{a}rding's inequality] \label{prop:garding:bound} Given $M \in S^0,$ there holds 
\be \label{up:g}
   \big| \exp\big( t \op_\e(M)\big) \big) \big|_{L^2 \to L^2} \lesssim \exp\big( \, t \big( \sup_{x,\xi} \sigma\big(\Re e\, M\big)+ \e C(M)\big)\,\big),
\ee
for all $t \geq 0,$ where $\Re e \, M$ is the hermitian matrix $\Re e \, M := (M + M^*)/2,$ and $C(M) > 0$ can be expressed in terms of a norm $\| M \|_{C(d)}$ of $M$ (see Sections {\rm \ref{sec:Fef-pho}} and {\rm \ref{sec:symbols}}). 
\end{prop}

\begin{proof}[Proof of Proposition \ref{prop:garding:bound}] Let $u \in L^2,$ and consider the solution $v$ to $\d_t v = \op_\e(M) v$ issued from $v(0) = u.$ Denoting $\bar \g := \sup_{x,\xi} \sigma\big(\Re e\, M\big),$ we let $w(t) = e^{- t \bar \g} v(t).$ Then, $w$ solves $\d_t w = \op_\e(M - \bar \g),$ so that $(1/2) \d_t |w|_{L^2}^2 = \Re e\, (\op_\e(M - \bar \g) u, u)_{L^2}.$ Since $\Re e \, (M - \bar \g) \leq 0,$ we may apply G\r{a}rding's inequality (Theorem {\rm \ref{th:G}})\footnote{Here we are using a semiclassical version of G\r{a}rding's inequality, in which the gain of one derivative in the remainder translates into a power of $\e.$ As discussed in Section \ref{sec:matrix}, our proof of Theorem \ref{th:G} yields for matrix-valued symbols a gain of one-half of a derivative, implying an error in $\e^{1/2} C(M)$ in \eqref{up:g}.}. This gives $\d_t |w|_{L^2}^2 \leq \e C(M) |w|_{L^2}^2,$ whence \eqref{up:g}. 
\end{proof}

When $M$ is hermitian, the bounds \eqref{up:g} and \eqref{up} coincide to first order. That is, if $M = \Re e \, M,$ then the growth rates are equal to first order: $\sup \Re e \, \s(M) = \sup \, \s(\Re e \, M).$ 

Otherwise, bound \eqref{up} is typically strictly sharper than \eqref{up:g}, as shown by example $\dsp{M = \left(\begin{array}{cc} 0 & a_{12} \\ a_{21} & 0 \end{array}\right),}$ for which $\max \Re e \, \s(M) = (a_{12} a_{21})^{1/2},$ assuming $a_{12} a_{21} > 0,$ and $\max \s(\Re e \, M) = (a_{12} + a_{21})/2,$ strictly greater than $(a_{12} a_{21})^{1/2}$ if $a_{12} \neq a_{21},$ by the arithmeti\-co-geometric inequality. In this example $M$ is symmetrizable. An example with a Jordan block is given by $M = \left(\begin{array}{cc} a & 1 \\ 0 & a \end{array}\right),$ where $a > 0,$ for which $\max \Re e \, \s(M) = a < \max \s(\Re e \, M) = a + 1/2.$

 \subsection{Application to instability} \label{sec:insta}

The semigroup bounds of Theorem \ref{th:spectral} translate into an instability result if the symbol $M$ has unstable spectrum. Consider the situation of a semilinear equation
\be \label{ideal}
 \d_t u = \op_\e(M) u + B(u,u),
\ee
where $M \in S^0$ and $B: \R^n \times \R^n \to \R^n$ is bilinear. All the results below are easily adapted to more general polynomial nonlinearities.  

\begin{lem} \label{lem:sup:norm} Given $u_0 \in H^s,$ with $s > d/2,$ equation \eqref{ideal} has a unique solution $u \in C^0([0, t_\star(\e)), H^s(\R^d))$ issued from $u_0,$ for some $t_\star(\e) > 0,$ which depends on $u_0$ and of course on $\e.$ The map $t \to |u(t)|_{L^\infty}$ is continuous in $[0, t_\star(\e)).$ The solution can be continued so long as its $\| \cdot \|_{\e,s}$ norm is finite. 
\end{lem}

\begin{proof} By bound \eqref{continuite:s} in the Appendix, the operator $\op_\e(M)$ is linear bounded $H^s \to H^s.$ For $s > d/2,$ the Sobolev embedding $H^s \hookrightarrow L^\infty,$ with $H^s$ endowed with norm $\| \cdot \|_{\e,s}$ defined in \eqref{semicl}, has norm $c_s \e^{-d/2},$ for some $c_s > 0.$ Thus there holds, by bilinearity of $B,$ for some $C_B > 0,$ 
 \be \label{buus} \|B(u,u) \|_{\e,s} \leq C_B |u|_{L^\infty} \| u \|_{\e,s} \leq C_B c_s \e^{-d/2} \|u \|_{\e,s}^2.\ee
 Equation \eqref{ideal} therefore appears as an ordinary differential equation in $H^s.$ The existence and uniqueness in $C^0([0, t_\star(\e)), H^s(\R^d)),$ for some $t_\star(\e) > 0,$ of a solution to \eqref{ideal} issued from $u_0 \in H^s,$ $s > d/2,$ follows from the Cauchy-Lipschitz theorem. The existence time $t_\star(\e) > 0$ a priori depends on $\e,$ and might be small. The map $t \to \| u(t)\|_{\e,s}$ is continuous, by the triangular inequality. For $[t, t+h] \subset [0, t_\star(\e)),$ there holds
 $$ |u(t+h) - u(t)|_{L^\infty} \leq c_s \e^{-d/2} \| u(t+h) - u(t)\|_{\e,s},$$
 implying continuity of $t \to |u(t)|_{L^\infty}$ in $[0, t_\star(\e)).$ Finally, bound \eqref{buus} and continuity of the linear solution operator $e^{t \op_\e(M)}$ in $\| \cdot \|_{\e,s}$ norm (estimate \eqref{up:s} in Theorem \ref{prop:s}) classically imply that elements of a given ball in $H^s$ generate solutions with a common existence time, and this implies that the solution can be continued so long as its $\| \cdot \|_{\e,s}$ norm is finite. 
\end{proof}

\begin{theo} \label{th:instability} If the spectrum of symbol $M$ is unstable, meaning
$$ {\rm sp}\, M(x,\xi) \cap \{ z \in \C, \, \Re e \, z > 0\} \neq \emptyset,\quad \mbox{for some $(x,\xi) \in \R^d \times \R^d$},$$ then for any $K > 0,$ some $P > 0,$  any $s > d/2,$ for $\e$ small enough, some $T_\star(\e) = O(|\ln \e|),$ we can find a datum 
$$\mbox{$u(0) \in C^\infty_c(\R^d),$ \,\, with \,\, $|u(0)|_{L^\infty} + |u(0)|_{\e,s} \lesssim \e^{K},$}$$
such that the solution $u$ to \eqref{ideal} issued from $u(0)$ belongs to $C^0([0,T_\star(\e)], H^s(\R^d)),$ and satisfies
\be \label{low:th} \sup_{0 \leq t \leq T_\star(\e)} |u(t)|_{L^\infty(\R^d)} \geq |\ln \e|^{-P}.\ee
\end{theo}

Thus we obtain a strong, albeit relative, instability under the mere assumption that the symbol has unstable spectrum. The instability is relative, in the sense that the deviation from the trivial solution depends on $\e.$ The deviation, however, is strong: it is expressed in terms of an inverse power of $|\ln \e|,$ starting from an initial amplitude that is an arbitrarily large power of $\e.$ In particular, the lower bound \eqref{low:th} implies $|u(T_\star(\e))|_{L^\infty} \geq \e^\a,$  for any $\a > 0$ and $\e$ small enough. By comparison, less-than-optimal semigroup bounds would only imply a lower bound in $\e^{\a(K)},$ for some $\a(K) > 0.$ We expand on this point after the proof of Theorem \ref{th:instability}. 

\begin{proof}[Proof of Theorem \ref{th:instability}] The parameters $K > 0$ and $s > d/2$ are given. We apply Theorem \ref{prop:s}: for $\g := \sup \Re e \, \s(M) > 0$ there holds
\be \label{after} 
 \begin{aligned} \big\| \exp \big( t \op_\e(M) \big) \, v \|_{\e,s} & \leq |\ln \e|^{N} e^{t \g} \| v \|_{\e,s} \\
  \big| \exp \big( t \op_\e(M) \big) \, v |_{L^\infty} & \leq |\ln \e|^{N} e^{t \g} | v |_{L^\infty} \big(1 + |\ln \e| + \ln \frac{\| v \|_{\e,s}}{|v|_{L^\infty}}\big), \end{aligned}\ee
  uniformly in $t \in [0, K |\ln \e|/\g],$ for $\e$ small enough and some $N > 0,$ depending on $M$ and $K,$ for all $v \in H^s.$ Here we do need to track down powers of $|\ln \e|,$ hence give up notation $|\ln \e|^*.$ 
  
  By Theorem \ref{th:spectral} and its proof, there holds also
  \be \label{low:after} |\ln \e|^{-N} \e^K e^{t \g} \leq \big| \exp \big( t \op_\e(M) \big) u_0 \big|_{L^2(B)},\ee
 for some ball $B = B(x_0, |\ln\e|^{-1}),$ and some $\e$-dependent family of data $u_0 \in H^s,$ such that 
 \be \label{after:u0}
 |u_0|_{L^\infty} = \e^K, \qquad \|u_0\|_{\e,s} = C_s \e^K,
 \ee
 for some $C_s > 0.$ 
Let the limiting observation time 
$$ %
T_\star(\e) := \frac{K}{\g} |\ln \e| - (4 N + 5) \ln |\ln \e|.
$$ %
With this choice there holds
\be \label{entails}
 \lim_{\e \to 0} \sup_{0 \leq t \leq T_\star(\e)} \e^K |\ln \e|^{4 ( N + 1)} e^{\g t}  = 0.
\ee
All the estimates below are valid for small enough $\e,$ depending only on $M,$ $K,$ $P,$ and the space dimensions. Denote, for $p \in \R,$ 
\be \label{y+y-}
 y_p(t) := |\ln \e|^p \e^K e^{t \g}.
 \ee
Denote also $u$ the solution issued from the datum $u_0$ described above. Its existence is granted by Lemma \ref{lem:sup:norm}. The maximal interval of existence is denoted $[0, t_\star(\e)).$ The function $t \to |u(t)|_{L^\infty}$ is continuous over $[0, t_\star(\e)).$ Let 
 $$ u_f(t) = e^{t \op_\e(M)} u_0, \qquad g(t) = \int_0^t e^{(t - t') \op_\e(M)} B(u(t'), u(t')) \, dt',$$
 so that $u = u_f + g.$ %

\medskip

 {\it Bounds for $u_f$.} By \eqref{after}(i), using notation introduced in \eqref{y+y-}, there holds
\be \label{free:s} \|u_f(t)\|_{\e,s} \leq C_s y_{N}(t), \quad \mbox{for some $C_s > 0.$}\ee
By \eqref{after}(ii) and \eqref{after:u0}, there holds, for some $C > 0,$ 
\be \label{free:pw} |u_f(t)|_{L^\infty} \leq C  y_{N + 1}(t).\ee

\medskip

{\it Propagation of upper bounds up to time $T_\star(\e).$} Consider the bounds
\be \label{at:t'}
\left\{\begin{aligned} \| u (t) \|_{\e,s} & \leq 4 C_s y_{N}(t) \\
|u(t)|_{L^\infty} & \leq 4 C  y_{N + 1}(t),
 \end{aligned}\right.
\ee
where $C > 0$ and $C_s > 0$ are as in \eqref{free:s} and \eqref{free:pw}.
 Let 
$$ I  = \Big\{ t \in [0, T_\star(\e)] \cap [0, t_\star(\e)), \quad \forall\, t' \in (0,t),\,\,\, \mbox{bounds \eqref{at:t'} hold at $t'$} \Big\}.$$
The goal is to show that $I = [0, T_\star(\e)] \cap [0,t_\star(\e)).$ The set $I$ is not empty by continuity of the $\| \cdot \|_{\e,s}$ and $|\cdot|_{L^\infty}$ norms of the solution (see Lemma \ref{lem:sup:norm}), and bounds \eqref{after:u0} for the initial datum. Besides, $I$ is closed, by construction. Indeed, if a sequence $(t_n) \subset I$ converges to $t_\infty \in [0, T_\star(\e)] \cap [0,t_\star(\e)),$ then we directly have $t_\infty \in I,$ unless $t_n < t_\infty$ for all $n.$ In this case, given $t' < t_\infty,$ then there exists $n_0$ such that $t' < t_{n_0} < t_\infty,$ and since $t_{n_0} \in I,$ this implies that bounds \eqref{at:t'} hold at $t'.$ 

We now prove that $I$ is open in $[0, T_\star(\e)] \cap [0,t_\star(\e)).$ Let $t \in I.$ There holds, by \eqref{after}(i) and \eqref{free:s},
 $$ \| u(t)\|_{\e,s} \leq C_s y_{N}(t) + \int_0^t |\ln \e|^{N} e^{(t - t') \g} \|B(u(t'),u(t'))\|_{\e,s} \, dt'.$$
By bilinearity of $B$ and bounds \eqref{at:t'}, for $t' \in [0,t],$ 
\be \label{up:Buu} \| B(u(t'), u(t')\|_{\e,s} \lesssim |u(t')|_{L^\infty} \| u(t') \|_{\e,s} \lesssim y_{N+1}(t') y_N(t').\ee
Thus
\be \label{for:I} \| u(t)\|_{\e,s} \leq C_s y_N(t) + C' y_{2 N + 1}(t) y_N(t), \quad \mbox{for some $C' > 0.$}\ee
But then $y_{2N + 1}(t)$ is small in the limit $\e \to 0,$ uniformly in $t \in [0, T_\star(\e)],$ by virtue of \eqref{entails}. Thus from \eqref{for:I} we deduce the bound at $t \in I:$ 
 $$ \| u(t)\|_{\e,s} \leq 2 C_s y_{N}(t),$$
 implying, by continuity of $\| u(\cdot)\|_{\e,s},$ that the upper bound \eqref{at:t'}(i) holds over a small time interval $[t, t+ h),$ for some $h > 0.$

Next, by \eqref{free:pw} and \eqref{after}(ii),
$$ |u(t)_{L^\infty} \leq C y_{N+1}(t) +  \int_0^t |\ln \e|^{N} e^{(t - t') \g} |B(u,u)(t')|_{L^\infty}\Big(1 + |\ln \e| + \ln \frac{\| B(u,u)(t')\|_{\e,s}}{|B(u,u)(t')|_{L^\infty}}\Big)\,dt'.$$
By \eqref{up:Buu}, there holds for $t' \in [0,t],$
$$ %
\big|\ln \| B(u,u)(t')\|_{\e,s} \big| \lesssim |\ln \e|.
$$ %
For the other logarithmic term, we use the elementary bound, for $x > 0,$ 
$$ x |\ln x| \lesssim x^{2/3} + x^2.$$
Since $B$ is bilinear, this implies  
$$  |B(u,u)(t')| \, \big|\ln |B(u,u)(t')|_{L^\infty} \big| \lesssim |u(t')|_{L^\infty}^{4/3} + |u(t')|_{L^\infty}^4.$$ 
Thus, with \eqref{at:t'}(ii), for $t' \in [0,t],$ for some $C' > 0,$ 
$$ |u(t)_{L^\infty} \leq C y_{N+1}(t) +  C' \big(y_{2 N+1}(t) + y_{4(N+1)}(t)^{1/3} + y_{4N/3 + 1}(t)^3 \big) y_{N+1}(t).$$
By \eqref{entails}, the function $y_{2N+1},$ $y_{4(N+1)}$ and $y_{4N/3+1}$ all converge to $0$ as $\e \to 0,$ uniformly in $t \in [0, T_\star(\e)].$ This proves
$$ |u(t)|_{L^\infty} \leq 2 C y_{N+1}(t),$$
for small enough $\e,$ and thus continuation of the a priori bound \eqref{at:t'} to the right of $t \in I.$ 

\medskip

{\it Conclusion.} The set $I$ is non empty, closed and open in $[0, T_\star(\e)] \cap [0,t_\star(\e)),$ thus equal to $[0, T_\star(\e)] \cap [0,t_\star(\e)).$ In particular, for fixed $\e,$ the function $t \to \|u(t)\|_{\e,s}$ is bounded in $I.$ By the continuation criterion of Lemma \ref{lem:sup:norm}, this implies $T_\star(\e) < t_\star(\e).$ Evaluating at $t = T_\star(\e),$ we find, by \eqref{low:after} and the same upper bounds as above,
$$ |u(T_\star(\e))|_{L^2(B)} \geq y_{-N}(T_\star(\e))(1 - o(1)),$$
where $o(1)$ is meant in the limit $\e \to 0.$ This implies \eqref{low:th}, with $P = 5 N + 5.$  
\end{proof}

 We conclude this Section by arguing that less-than-optimal semigroup bounds, such as given by G\r{a}rding's inequality, as seen in Proposition \ref{prop:garding:bound}, imply a much  weaker form of instability for the trivial solution to \eqref{ideal}. 
 
\begin{rem} \label{rem:rem} Given $M \in S^0$ with unstable spectrum, as in Theorem {\rm \ref{th:instability}}, if instead of the upper and lower bounds \eqref{after} and \eqref{low:after} for the group of operators $\exp(t \op_\e(M)),$ we had similar bounds with an {\it upper} rate of exponential growth $\bar \g$ and a {\it lower} rate of exponential growth $\g,$ such that $ \g < \bar \g,$ then the result of Theorem {\rm \ref{th:instability}} would still hold, but only with the weaker deviation estimate  
 \be \label{holder} \sup_{\begin{smallmatrix} 0 < \e  < 1 \\ 0 \leq t \leq T_\star(\e) \end{smallmatrix} } \frac{|u(t) |_{L^\infty}}{|u(0)|_{L^\infty}^{\b}} = \infty, \quad T_\star(\e) = O(|\ln \e|), \quad \mbox{for $\dsp{\b > 1 - \frac{\g}{\bar \g}.}$}\ee
 \end{rem}
 
 \begin{proof}[Verification of the claim in Remark \ref{rem:rem}] It suffices to adapt the proof of Theorem \ref{th:instability}. Disregarding powers of $|\ln \e|,$ the goal is to compare the free solution $u_f$ to a Duhamel term bounded in $\| \cdot \|_{\e,s}$ by $\dsp{\int_0^t e^{\bar \g (t - t')} |u|_{L^\infty} \|u\|_{\e,s} \, dt'.}$ Existence is granted in time $\dsp{\frac{K}{\bar \g} |\ln \e|}.$ The free solution is bounded from below by $\e^K e^{t \g},$ thus dominates the Duhamel term only so long as $|u|_{L^\infty} \leq \e^\a$ and $t \leq \a |\ln \e|/(\bar \g - \g),$ and is greater than $\e^\a$ in time $(K - \a) |\ln \e|/\g.$ We conclude that we have a proof of a deviation $|u|_{L^\infty} \geq (1/2) \e^\a$ from $|u(0)|_{L^\infty} = \e^K$ if $(K - \a)/ \g < \a/(\bar \g - \g).$ This translates into \eqref{holder}. 
\end{proof}
 %
  
%If, for instance, $\bar \g < 2 \g,$ the deviation estimate \eqref{holder} indicates a lack of H\"older estimate for the flow of \eqref{ideal} in time $O(|\ln \e|).$ If $\bar \g \geq 2 \g,$ then estimate \eqref{holder} is weaker than a lack of a Lipschitz estimate for the flow. 

%\smallskip

 By comparison, Theorem \ref{th:instability} implies that there holds \eqref{holder} {\it for any $\b > 0,$} 
 and also the stronger estimate
 $$ \sup_{0 < \e < 1} \frac{ \ln |u(T_\star(\e))|_{L^\infty} }{\ln \big| \ln |u(0)|_{L^\infty} \big|}  \geq - P,$$
 where $P$ depends on $K,$ with $|u(0)|_{L^\infty} = \e^K.$

\section{Application: a new proof of sharp G\r{a}rding inequalities} \label{sec:Fef-pho}

We prove here the following G\r{a}rding inequalities:

\begin{theo} \label{th:G} For all $m \in \R,$ for all scalar symbol $a \in S^m(\R^d \times \R^d)$ such that $\Re e \, a \geq 0,$ for all $0 < \theta < 1,$ for some $c > 0,$  
 there holds for all $u \in H^m(\R^d)$ the lower bound 
$$ %
  \Re e \, \big( \op(a) u, u \big)_{L^2} + c |u |_{H^{(m-\theta)/2}}^2 \geq 0.
  $$ %
\end{theo}

The constant $c$ depends on $\theta$ and on a large number $r(\theta)$ of derivatives of $a,$ with $c \to \infty$ and $r \to \infty$ as $\theta \to 1.$ This highlights two shortcomings of Theorem \ref{th:G} and its proof: we do not  
handle the endpoint case $\theta = 1$ corresponding to the classical G\r{a}rding inequality (first proved by H\"ormander \cite{H} for scalar symbols, and extended to systems by Lax and Nirenberg \cite{LaxN}), and we require a lot of smoothness for $a.$

Nonetheless our proof may have some interest in its own right. First, it completely differs from the classical proofs, which  go either by reduction to the elliptic case (see for instance the proofs of Theorem 4.32 in \cite{Z} or Theorem 7.12 in \cite{DS}), or by use of the Wick quantization (see for instance the proof of Theorem 1.1.26 in \cite{Le}). Second, it lends itself to partial extensions, in particular to the matrix case, as discussed in Section \ref{sec:matrix}. Finally, it allows to view the Approximation Lemma \ref{lem:tildeS} as a refinement of G\r{a}rding's inequality, in the sense that Lemma \ref{lem:tildeS} implies G\r{a}rding (as shown by the proof below), and also implies stronger semigroup bounds than G\r{a}rding, as we saw in Section \ref{sec:ex}.

The proof of Theorem \ref{th:G} is given in Sections \ref{sec:red} to \ref{sec:4}. The key idea of the proof is the reformulation, in Section \ref{sec:red3}, of the G\r{a}rding inequality as an upper bound for the backwards flow of $a^{\rm w}$ (Weyl quantization). This is exploited in Section \ref{sec:approx} where we approximate the flow of $a^{\rm w},$ following the ideas of Section \ref{sec:Duhamel}. Estimates conclude the proof in Section \ref{sec:4}. 

 \subsection{First step: reductions} \label{sec:red}

We denote $a^{\rm w}$ the pseudo-differential operator in Weyl quantization (see definition \eqref{def:weyl} in the Appendix) with symbol $a.$ There holds  
$\op(a) = a^{\rm w} + \op(\rho),$
where $\rho \in S^{m-1},$ with norms bounded by norms of $a.$ In particular, there holds the bound $(\op(\rho) u, u)_{L^2} \leq \| a \|_{C(d)} |u|_{H^{(m-1)/2}}^2,$ for all $u \in H^{m-1}.$ Thus we may switch to a Weyl quantization. The adjoint of $a^{\rm w}$ is $(\bar a)^{\rm w},$ so that $\Re e \, (a^{\rm w} u, u)_{L^2} = \big( (\Re e \, a)^{\rm w} u, u)_{L^2}.$ Thus it suffices to handle the case $a \in \R.$ The goal is now to prove 
 \be \label{0:g} 
 \begin{aligned}  ( a^{\rm w} u, u)_{L^2} & + c |u|_{H^{(m-\theta)/2}}^2 \geq 0, \\ &  \mbox{for any real nonnegative $a \in S^m,$ some $c > 0,$ all $u \in H^{m}.$} \end{aligned}
 \ee
Let $\Lambda = \op(\langle \cdot \rangle),$ with $\langle \xi \rangle := (1 + |\xi|^2)^{1/2},$ and $a_0 := \langle \xi \rangle^{-m} a \in S^0_{1,0}.$ Consider the operator $\Lambda^{m/2} a_0^{\rm w} \Lambda^{m/2}.$ Its principal symbol is $a.$ In Weyl quantization, its subprincipal symbol is 
$$ %
\frac{1}{2i} \{ \langle \xi \rangle^{m/2},  a_0 \langle \xi \rangle^{m/2} \} + \frac{1}{2i} \langle \xi \rangle^{m/2} \{ a_0, \langle \xi \rangle^{m/2} \} = 0.
$$ %
  Hence, by composition of operators (see \eqref{compo:w} and \eqref{composition:w}), there holds
$ \Lambda^{m/2} a_0^{\rm w} \Lambda^{m/2} = a^{\rm w} + R_{m-2},$
 where $(R_{m-2} w, w)_{L^2} \lesssim \| a \|_{C(d)} |w|_{H^{(m-2)/2}}^2,$ for all $w \in H^{(m-2)/2}.$  Since $\Lambda$ is $L^2$-self-adjoint, $(a^{\rm w} u, u)_{L^2} =  (a_0^{\rm w} \Lambda^{m/2} u, \Lambda^{m/2} u)_{L^2} +  (R_{m-2} u, u)_{L^2}.$
From the above, it appears that it is sufficient to prove \eqref{0:g} in the case $m = 0.$

Let $(\phi_j)_{j \geq 0}$ and $(\psi_j^2)_{j \geq 0}$ be two dyadic Littlewood-Paley decompositions, such that $\phi_j \equiv \phi_j \psi_j^2.$ Then there holds  (this is Claim 2.5.24 in \cite{Le})
\be \label{lp} \Big| (a^{\rm w} u, u)_{L^2} -  \sum_{j \geq 0} \big( (\phi_j a)^{\rm w} \psi_j^{\rm w} u,\, \psi_j^{\rm w} u\big)_{L^2} \Big| \leq c |u|_{H^{-1}}^2,\ee
 where $c$ depends on norms of $a,$ and $\psi_j^{\rm w}$ is the Fourier multiplier $\op(\psi_j).$ Thus it suffices to prove
\be \label{0.g}
 \big( a_j^{\rm w} u_j, u_j\big)_{L^2} + c 2^{- j \theta} |u_j|_{L^2}^2 \geq 0, \qquad a_j = \phi_j(\xi) a(x,\xi), \,\, u_j := \psi_j^{\rm w} u,
\ee 
for some $c > 0$ independent of $j,$ for all $j.$ In particular, we may change $a_j$ into $b_j = a_j + 2^{-j \theta}.$ Indeed, if we manage to prove \eqref{0.g} with $b_j^{\rm w}$ in place of $a_j^{\rm w},$ then \eqref{0.g} with holds for $a_j^{\rm w}$ as well, with the constant $c + 1.$ For notational simplicity, we keep notation $a_j$ to denote $a_j + 2^{-j \theta};$ in other words we now  assume $a_j \geq 2^{- j \theta},$ for all $j.$

We note moreover that low-frequency terms can be absorbed in the remainder, via
  $$ \sum_{0 \leq j \leq j_0} \big( a_j^{\rm w} u_j, u_j\big)_{L^2} \lesssim \| a \|_{0} 2^{j_0 \theta} |u|_{H^{-\theta/2}}^2,$$
  a consequence of the $L^2$ continuity of the $a_j$ (see \eqref{continuity:w}). 
  
 Finally, up to dividing by $\| a \|_{C(\theta)},$ where $C(\theta)$ is large enough, we may assume that a large number of norms of $a$ are bounded by 1. 

\smallskip
  
   In accordance with the above, in the rest of this proof a family of symbols $(a_j)$ is given, such that 
\be \label{assume}
 \mbox{$a_j$ is real}, \,\, a_j \in S^0, \,\, a_j \geq 2^{-j \theta}, \,\, \mbox{supp}\,\d_{x,\xi}^\a a_j \subset \{ (x,\xi),\ \,\, |\xi| \sim 2^j \}, \,\, \| a_j \|_{C(\theta)} \leq 1,
\ee 
with $C(\theta)$ possibly large, and we undertake to find $j_0 \in \N$ such that for all $u \in L^2,$ all $j \geq j_0,$ 
\be \label{00}
 \big( a_j^{\rm w} u_j, u_j\big)_{L^2} \geq 0, \qquad u_j := \psi_j^{\rm w} u.
 \ee
In the fourth condition in \eqref{assume}, $\mbox{supp}\, \d_{x,\xi}^\a a_j$ denotes support of $\d_{x,\xi} a_j$ for any $|\a| > 0;$ by $\{ (x,\xi), \,\, |\xi| \sim 2^j\}$ we mean the cartesian product of $\R^d_x$ with some annulus $\{ \xi \in \R^d, \,\, A 2^{j} \leq |\xi| \leq B 2^j\},$ with $0 < A < B$ independent of $j.$  We may simply think of $a_j$ as being $\phi_j a + 2^{-j \theta},$ where $a$ is a given symbol in $S^0,$ real and nonnegative, and $(\phi_j)$ is a Littlewood-Paley decomposition of unity.  

\subsection{Second step: reformulation in terms of the flow of $a_j^{\rm w}$} \label{sec:red3} 

By the Calder\'on-Vaillancourt theorem \eqref{continuity:w}, the operator $a_j^{\rm w}$ is linear bounded $L^2 \to L^2.$ Let $\Phi$ be the flow of $a_j^{\rm w}:$ $\Phi(t) = \exp(t a_j^{\rm w}),$ meaning that for all $w \in L^2,$ for all $t \in \R,$ $\Phi(t) w$ is the unique solution in $L^2$ to the initial-value problem
$$ %
  y' = a_j^{\rm w} y, \quad y(0) = w.
$$ %
 We compute, for $t \in \R$ and $w \in L^2,$ 
 \be \label{flow2} \frac{1}{2} \frac{d}{dt} \big| \Phi(t) w\big|_{L^2}^2 =  \big( a_j^{\rm w}\Phi(t) w, \Phi(t) w \big)_{L^2},\ee
 and
\be \label{flow3} 
 \begin{aligned}
 \frac{1}{2} \frac{d}{dt} \big(a_j^{\rm w} \Phi(t) w, \Phi(t) w \big)_{L^2} 
 & =  \big| a_j^{\rm w} \Phi(t) w\big|_{L^2}^2 \geq 0,
 \end{aligned}
 \ee
 so that the right-hand side in \eqref{flow2} is a growing function of time.
Integrating \eqref{flow2} from $0$ to $t,$ we find%
$$ %
 |\Phi(t) w|_{L^2}^2 - | w |_{L^2}^2  = 2 \int_0^{t}  \big( a_j^{\rm w} \Phi(t') w, \Phi(t') w \big)_{L^2} \, dt', 
$$ %
 hence with \eqref{flow3}, the inequality 
 \be \label{flow5}
 |\Phi(t) w|_{L^2}^2 - | w |_{L^2}^2 \leq 2 t \, \big( a_j^{\rm w} \Phi(t) w, \Phi(t) w\big)_{L^2}, \quad t \geq 0, w \in L^2.
  \ee
For all $t,$ the operator $\Phi(t)$ is onto $L^2$ (indeed, there holds $\Id_{{\mathcal L}(L^2)} = \Phi(t) \Phi(-t)$), so that, for $u_j$ defined in \eqref{00}, we can write $u_j = \Phi(t) w$ with $w = \Phi(-t) u_j,$ and \eqref{flow5} becomes
$$ %
  | u_j |_{L^2}^2 - | \Phi(-t) u_j |_{L^2}^2 \leq 2 t \, \Re e \, \big( a_j^{\rm w} u_j, u_j\big)_{L^2}.
$$ %
 Thus, in order to prove \eqref{00}, it is sufficient to show that for some $j_0 \geq 0,$ all $u \in L^2,$ all $j \geq j_0,$ for some $t > 0,$ there holds
 \be \label{ineq:to-prove}
  \big| \Phi(-t)  u_j \big|_{L^2} \leq | u_j|_{L^2},
  \ee  
 for $a$ satisfying \eqref{assume}, with $u_j = \psi_j^{\rm w} u,$ where $(\psi_j)$ is a Littlewood-Paley partition of unity such that \eqref{lp} holds.
 
  At this stage we have reformulated the G\r{a}rding inequality \eqref{00} into the upper bound \eqref{ineq:to-prove} for the backward flow $\Phi(-t)$ of $a_j^{\rm w}.$ 

\subsection{Third step: approximation of the flow of $a_j^{\rm w}$} \label{sec:approx}

We denote $S_0 := e^{-t a_j},$ and define correctors $(S_q)_{1 \leq q \leq q_0}$ by
\be \label{def:Sq} \d_t S_q = - a_j S_q - \sum_{\begin{smallmatrix} q_1 + q_2 = q \\ 0 < q_1 \end{smallmatrix} } a_j \diamond_{q_1} S_{q_2}, \qquad S_q(0) = 0,
\ee
with notation $\diamond$ introduced in \eqref{diamond}.
\begin{lem} \label{lem:Sa} For $\Sigma  := \sum_{0 \leq q \leq q_0-1} S_q,$ for $\dsp{ q_0:= 1 + c_d \Big[ \frac{\theta}{1 - \theta}\Big]},$ for some $c_d > 0$ depending only on $d,$ there holds
 \be \label{approx:GFP}
 \d_t \Sigma^{\rm w} = - a_j^{\rm w} \Sigma^{\rm w} + \rho(t)^{\rm w},
 \ee
  where, for any $s \in \R,$ 
   $$| \rho(t)^{\rm w} w|_{H^s} \lesssim \s(t) |w|_{H^{s-q_0}}, \qquad \s(t) := \sum_{0 \leq q \leq q_0 - 1} \| S_q(t) \|_{q_0 - q + C(d)},$$
for some $C(d) > 0$ depending only on $d.$ 
 \end{lem}

 The reason for our choice of $q_0$ will be apparent after Lemma \ref{lem:P}. 

\begin{proof} By exactly the same computations as in the proof of the Approximation Lemma \ref{lem:tildeS}, we find that \eqref{approx:GFP} holds with $\rho =  \sum_{0 \leq q \leq q_0-1} R^{\rm w}_{q_0 - q }(a_j, S_q) \in S^{-q_0}.$ The bound for $\rho^{\rm w}$ derives from \eqref{sobolev:w}. 
\end{proof}

\begin{cor} \label{cor:duh} For some $C, C' > 0$ depending only on $d,$ so long as \be \label{t:j} C t 2^{-j q_0} | \s |_{L^\infty(0,t)} < 1/2,
\ee
there holds the bound%
$$%
  |\Phi(-t) u_j|_{L^2} \leq  | \Sigma(t)^{\rm w} u_j|_{L^2} + C' t^2 2^{-j q_0} |\s|_{L^\infty(0,t)} \big| \| \Sigma \|_{0} \big|_{L^\infty(0,t)} |u_j|_{L^2}.
$$ 
 \end{cor}

\begin{proof} We follow the proof of Theorem \ref{th:duh}, but here we do not seek here a representation of the whole flow, only of its action on $u_j.$ We deduce from Lemma \ref{lem:Sa} the representation
\be \label{rep:Phi} \Phi(-t) u_j = \Sigma(t)^{\rm w} u_j - \int_0^t \Sigma(t -t')^{\rm w} \sum_{k \geq 0} (-1)^{k+1} \rho_0^k \Big(\rho(\cdot)^{\rm w} u_j\Big)(t') \, dt',\ee
where %
$\dsp{(\rho_0 w)(t) := \int_0^t \rho(t - t')^{\rm w} w(t') \, dt'.}$
 By Lemma \ref{lem:Sa} and a straightforward induction,  %
$$ \big|\rho_0^k \big( \rho(\cdot)^{\rm w} u_j\big)(t)\big|_{L^2} \lesssim \big(t 2^{-j q_0} |\s|_{L^\infty(0,t)} \big)^{k+1} |u_j|_{L^2},$$
using the frequency localization of $u_j.$ 
From there we deduce that the sum in \eqref{rep:Phi} converges if $t 2^{-j q_0} |\s|_{L^\infty(0,t)}$ is small enough, depending only on $d.$ %
\end{proof}

Recall that the goal is to prove \eqref{ineq:to-prove}. According to Corollary \ref{cor:duh}, it is sufficient to find $t$ such that \eqref{t:j} holds, and also
\be \label{to:prove2} \begin{aligned} 
 C \| \Sigma(t) \|_{0} + C t^2 2^{-j q_0} |\s|_{L^\infty(0,t)} \big| \| \Sigma \|_{0} \big|_{L^\infty(0,t)} \leq 1
\end{aligned}
\ee
for $C > 0$ depending only on $d.$ %

  \subsection{Fourth step: final estimates} \label{sec:4}

The observation time is set to 
\be \label{t:final}
 t_\star :=  j \t_\star 2^{j \theta},
 \ee
 with $\t_\star > 0$ depending only on $d,$ to be chosen large enough below. 
 
\begin{lem} \label{lem:P} For $0 \leq t \leq t_\star,$  for all $0 \leq q \leq q_0,$ all $\a, \b,$ there holds  %
 \be \label{est:Sq} \big| \langle \xi \rangle^{q + |\b|} \d_x^\a \d_\xi^\b S_q(t)\big|_{L^\infty} \leq P_j \big( 1 + t \big)^{q + (|\a| + |\b|)/2} \exp\big(-t 2^{-j \theta}\big),\ee
 where $P_j$ is a polynomial in $j,$ of degree $q + (|\a| + |\b|)/2.$ 
\end{lem}

\begin{proof} {\it First step.} We claim that on $\{ a_j < h\},$ there holds $|\tilde D a_j| \leq 4 h^{1/2},$ where 
$$\tilde D = (\nabla_x, \langle \xi \rangle \nabla_\xi).$$

 Indeed, let $(x,\xi) \in \{ a_j < h \},$ let $\vec e$ be a given unitary direction in $\R^{d}_x,$ and $\big[x - s_- \vec e, x + s_+ \vec e \, \big] \times \{ \xi \}$ be the line segment of maximal (and temporarily assumed finite) length in $\{ a_j < h\}$ that goes through $(x,\xi)$ and is parallel to $\vec e.$ By maximality of the segment and continuity of $a,$ the function 
 $\tilde a(s) := a_j\left(x + s \vec e, \xi\right)$ cannot be monotonous in $[s_-,s_+].$ In particular, for some $s_0$ there holds $\tilde a'(s_0) = 0.$ If $|s_+ - s_-| \leq 2 h^{1/2},$ then this implies the bound on $|\nabla_x a|,$ since $|\d_x^2 a| \leq 1$ and $\vec e$ is arbitrary. Otherwise, Taylor expansions imply $$ \tilde a'(s) = \frac{\tilde a(s + h^{1/2}) - \tilde a(s - h^{1/2})}{2 h^{1/2}} + h^{1/2} \int_0^1 \tilde a''(s + h^{1/2} \s) - \tilde a''(s - h^{1/2} \s) \, d\s,$$
 and given $s \in \big(s_- + h^{1/2}, s_+ - h^{1/2}\big),$ we may bound $\tilde a(s \pm h^{1/2})$ by $h.$ This implies $|\tilde a'(s)| \leq 3 h^{1/2},$ since $|a''| \leq 1.$ Finally on $(s_-, s_- + h^{1/2}],$ we simply use another first-order Taylor expansion of $\tilde a',$ and the fact that $|\tilde a'(s_-+h^{1/2})| \leq 3 h^{1/2}.$ The same argument applies on $[s_+ - h^{1/2}, s_+).$ If the considered line segment is infinite, a minor variation on the above arguments applies. This proves the bound on $|\nabla_x a|.$ For the bound on $\langle \xi \rangle |\nabla_\xi a|,$ it suffices to consider a line segment parallel to a direction in $\R^d_\xi,$ and use $\langle \xi \rangle^2 |\d_\xi^2 a| \leq 1.$ Here the discussion bears on whether the length of the segment is smaller or greater than $2 h^{1/2} \langle \xi \rangle.$ 

\smallskip

{\it Second step.} By the Fa\'a di Bruno formula, denoting 
\be \label{def:tildeDg} \tilde D^\g = \langle \xi \rangle^{|\b|} \d_x^\a \d_\xi^\b, \qquad \mbox{with $\a + \b = \g,$}\ee
there holds
 \be \label{dke} \tilde D^\g(e^{-ta_j}) =  e^{-t a_j} \sum_{\begin{smallmatrix} 1 \leq k \leq |\g| \\ \a_1 + \dots + \a_{k} = \g \end{smallmatrix}} C_{(\a_1,\dots,\a_k)} t^{k} \prod_{1 \leq \ell \leq k} \tilde D^{\a_\ell} (-a_j),\ee
where $C_{(\a_\ell)}$ are positive constants. Let $0 \leq k_0 \leq k$ such that $|\a_\ell| = 1$ if $\ell \leq k_0.$ Since the other indices $\a_\ell$ all have length greater than two, and since there are $k - k_0$ of them, there holds $|\g| \geq k_0 + 2 (k - k_0).$ We thus obtain
\be \label{dec:faa} \tilde D^\g (e^{-t a_j}) = e^{-t a_j} \sum C_{\star} t^{k} (\tilde D a_j)^{k_0} P_{\star}(\d)(\tilde D^2 a), \qquad 2 k \leq |\g| + k_0, \quad k_0 \leq |\g|,\ee
where $(\tilde D a_j)^{k_0} =  \prod_{1 \leq i,j \leq d} (\d_{x_i} a)^{\g_0^{(i)}}(\langle \xi \rangle \d_{\xi_j} a)^{\g_0^{(j)}},$ for some $\g_0 \in \N^{2d}$ such that $|\g_0| = k_0,$ and $P_{\star}$ is a constant-coefficient polynomial, so that $P_\star(\d) (\tilde D^2 a)$ involves only (weighted) derivatives of $a$ of order at least two. In \eqref{dec:faa}, the sum runs over all possible decompositions of $\g$ as in \eqref{dke}, and the $C_\star$ are positive constants.

\smallskip

{\it Third step.} We now verify by induction that for all $\g,$ all $q \leq q_0,$ 
\be \label{ind:Sq}
\langle \xi \rangle^q \tilde D^\g S_q = e^{-t a_j} \sum C_\star t^{k} (\tilde Da_j)^{k_0} P_\star(\d)(\tilde D^2 a),
\ee
with the same summation convention as in \eqref{dec:faa}, and 
\be \label{ind:Sq2} \max(k_0,k) \leq 2 q + |\g|, \qquad  \dsp{k - k_0/2 \leq q + |\g|/2}.\ee
Recall that in \eqref{ind:Sq}, $\tilde D^\g$ is a weighted derivative (it is defined in \eqref{def:tildeDg}), so that the total weight in the left-hand side of \eqref{ind:Sq} is $\langle \xi \rangle^{q + |\b|},$ with $\g = \a + \b,$ as in \eqref{est:Sq}.  

For $q = 1,$ there holds $S_1 = 0,$ by \eqref{diamond}. For $q = 2,$ $\tilde D^\g S_2$ is a sum of terms of the form $t^2 P_\star(\tilde D^2 a_j) \tilde D^\g e^{-t a_j},$ and of terms of the form $t^3 P_\star(\tilde D^2 a_j) \tilde D^\g ((\tilde D a_j)^2 e^{- t a_j}).$ In both cases, we verify conditions \eqref{ind:Sq}-\eqref{ind:Sq2} directly, using the second step.

Suppose now that \eqref{ind:Sq} holds for all $q' \leq q - 1.$ By definition of $S_q$ in \eqref{def:Sq}, $ \langle \xi \rangle^q \tilde D^\g S_q$ is a sum of terms
$$ %
 \int_0^t \tilde D^{\g_1} \big( e^{(t - t') a_j} \big) \tilde D^{\g_2 + q_1} a_j \langle \xi \rangle^{q_2}\tilde D^{\g_3 + q_1} S_{q_2}(t') \, dt', \quad 0 < q_1, \,\, q_1 + q_2 = q, \,\, |\g_1| + |\g_2| + |\g_3| = |\g|,
$$ 
By \eqref{dec:faa} and the induction hypothesis, up to multiplication by $C_\star P_\star(\d) (\tilde D^2 a_j)$ every term above is a sum of terms of the form 
$e^{-t a_j} t^{1 + k+ k'} (\tilde D a_j)^{k_0 + k'_0} \tilde D^{\g_2 + q_1} a,$
with
 $$2 k \leq |\g_1| + k_0, \quad k_0 \leq |\g_1|, \quad k' \leq 2 q_2 + |\g_3| + q_1, \quad k' - \frac{k'_0}{2} \leq q_2 + \frac{|\g_3|}{2} + \frac{q_1}{2}. $$
From there, we see that \eqref{ind:Sq} holds at rank $q,$ handling the case $|\g_2| + q_1 \leq 1$ separately.

\smallskip

{\it Fourth step.} For $0 \leq t \leq 2,$ the bound \eqref{est:Sq} follows from the previous step. We assume $t \geq 2$ from now on, and use the bound on $S_q$ given by the third step. 

 On $\{ a_j \geq 2^{-j \theta} + C t^{-1} \ln t\},$ bounding derivatives of $a$ by 1 and using \eqref{ind:Sq}-\eqref{ind:Sq2}, we find that there holds $\langle \xi \rangle^q |\tilde D^\g S_q| \leq C_q t^{2 q + |\g| - C} e^{-t 2^{- j \theta}},$ implying \eqref{est:Sq} if $C \geq 2 q + |\g|.$ 
 
 On $\{a_j < 2^{-j \theta} + C t^{-1} \ln t\},$ there holds $|\tilde D a_j| \leq 4 \big(2^{-j \theta} + C t^{-1} \ln t\big)^{1/2},$ by the first step. On $[0, t_\star],$ with the limiting observation time $t_\star$ as defined in \eqref{t:final}, there holds $2^{-j \theta} \leq C t^{-1} \ln t$ if $C$ is large enough (independently of $j$). Hence the bound $|\tilde D a_j| \leq 4 (2C)^{1/2} (t^{-1} \ln t)^{1/2}.$ Thus with \eqref{ind:Sq}-\eqref{ind:Sq2}, we find $\langle \xi \rangle^q |\tilde D^\g S_q| \leq C_q t^{k - k_0/2} (\ln t)^{k_0/2} e^{-t 2^{-j \theta}},$ implying \eqref{est:Sq}, since $(\ln t)^{k_0/2} \lesssim j^{q +|\g|/2}.$  
 \end{proof}

By the Calder\'on-Vaillancourt theorem (bound \eqref{continuity:w} in Section \ref{sec:symbols}), 
$$ %
|S_q(t)^{\rm w}|_{L^2 \to L^2} \lesssim \sup_{\a,\b}  \big| \langle \xi \rangle^{|\b|} \d_x^\a \d_\xi^\b S_q(t) \big|_{L^\infty},
$$ %
 with $|\a|, |\b| \leq [d/2] + 1.$ We now use Lemma \ref{lem:P}. Since the correctors $S_q,$ for $q \geq 1,$ are localized around frequencies $\sim 2^j,$ and since $\max_{t \geq 0} t^k e^{-t 2^{-j \theta}} = C_k 2^{j \theta k},$ we obtain
\be \label{est:Sq-bis} \begin{aligned} \max_{0 \leq t \leq t_\star} | S_q(t)^{\rm w}|_{L^2 \to L^2} & \lesssim P_j 2^{-j q + j \theta(q + C(d))}, \\ | S_q(t_\star)^{\rm w}|_{L^2 \to L^2} & \lesssim P_j 2^{-j q + j \theta(q + C(d))} e^{-\t_\star j}, \\ \max_{0 \leq t \leq t_\star} \| S_q(t) \|_{q_0 + q + C(d)} & \lesssim P_j 2^{j \theta (q_0 + C(d))},
\end{aligned}
\ee
where $P_j$ is a polynomial in $j,$ of degree less than $q_0 + C(d),$ and $t_\star$ is the limiting observation time defined in \eqref{t:final}. Since $\theta < 1,$ we may sum the bounds in \eqref{est:Sq-bis} over $q,$ implying 
$$ %
\begin{aligned}
 \max_{0 \leq t \leq t_\star} | \Sigma(t)^{\rm w}|_{L^2 \to L^2} & \lesssim P_j 2^{j \theta C(d)}, \\ | \Sigma(t_\star)^{\rm w}|_{L^2 \to L^2} & \lesssim P_j 2^{j \theta C(d)} e^{-\t_\star j}, \\ |\s|_{L^\infty(0,t)} & \lesssim P_j 2^{j \theta (q_0 + C(d))}. \end{aligned}
 $$ %
This shows that for $\t_\star$ large enough the bound \eqref{to:prove2} holds at $t = t_\star.$ Indeed, the first term in \eqref{to:prove2} is 
$$ C \| \Sigma(t_\star) \|_0 \leq C' P_j 2^{j \theta C(d)} e^{-\t_\star j} \leq 1/2,$$
if $\t_\star > \theta C(d) \ln 2,$ and if $j$ is large enough, depending on the degree $q_0 + C(d)$ of $P_j.$ And, by choice of $q_0$ in Lemma \ref{lem:Sa}, the second term in \eqref{to:prove2} is 
$$ C t_\star^2 2^{-j q_0} |\s|_{L^\infty(0,t_\star)} \big| \| \Sigma \|_{0} \big|_{L^\infty(0,t)} \leq C' P_j 2^{j \theta(q_0 + C(d)) - j q_0} \leq 1/2.$$
if $j$ is large enough, depending only on $\theta$ and $d.$ This concludes the proof of Theorem \ref{th:G}. %

\subsection{Remarks and extensions} \label{sec:matrix}

It is only in the first step of the proof of Lemma \ref{lem:P} that we use the assumption that $a$ is scalar. There we take advantage of the fact that if $a \in C^2$ is nonnegative, then $|\nabla a| \lesssim |a|^{1/2}$ in a neighborhood of $\{ a = 0\}.$ This implies that the correctors $S_q$ in the approximate solution operator do not grow in time like $t^{2q + C(d)},$ but only like $t^{q + C(d)}.$ Considering that our construction of the order-$q_0$ solution operator is accurate only for $t$ such that $t 2^{-j q_0} \s(t) < 1$ (this is Corollary \ref{cor:duh}) with $\s$ growing in time like $S_{q_0},$ this gives the constraint $2^{-j q_0} t^{q_0 + C(d)} < 1,$ implying for the limiting observation time $t_\star$ the bound $t_\star = O(2^{j \theta}),$ with $\theta < 1.$

Now for matrix-valued symbols, we have no such bound on $|\nabla a|.$ As a consequence, the correctors a priori grow like $t^{2q + C(d)}.$ Our proof thus adapts to matrix-valued symbols, but only if we restrict to $\theta < 1/2,$ corresponding to a gain of (just less than) half a derivative in G\r{a}rding.

Finally, we note that for operators in Weyl quantization, both the reductions to symbols of order zero and the Littlewood-Paley decomposition \eqref{lp} generate errors that are $O(|u|_{H^{-1}}^2).$ Thus the analysis of Section \ref{sec:red3} applies to the Fefferman-Phong inequality (\cite{FP,Bony}; Theorem 2.5.10 in \cite{Le}), a refinement of G\r{a}rding with gain of two derivatives, for scalar symbols:

\begin{prop} \label{prop:FP} In order to prove the Fefferman-Phong inequality
$$ %
 \Re e \, (a^{\rm w} u, u)_{L^2} + C |u|_{H^{(m-2)/2}}^2 \geq 0,
$$ %
known to hold for all scalar $a \in S^m$ with $\Re e \, a \geq 0,$ some $C > 0,$ all $u \in H^m,$ it is sufficient to prove that for all real $a \in S^0$ such that $a \geq 0,$ the following holds: for some $C > 0,$ for $j$ large enough, for all $u \in L^2,$ there holds for some $t > 0$ the bound 
\be \label{for:fp}
 |\Phi(-t) u_j|^2_{L^2} \leq |u_j|^2_{L^2} + C t \b_j, \qquad \mbox{with \,\, $\dsp{\sum_{j \geq j_0} \b_j \leq |u|_{H^{-1}}^2}.$}
\ee
where $\Phi$ is the flow of $2^{-j} + (\phi_j a)^{\rm w},$ $u_j = \psi_j^{\rm w} u,$ and $(\phi_j)$ and $(\psi_j^2)$ are two Littlewood-Paley decompositions such that $(1 - \psi_j^2) \phi_j \equiv 0.$ 
\end{prop}

For a proof of Proposition \ref{prop:FP}, it suffices to follow the reductions steps of Section \ref{sec:red} and reproduce the analysis of Section \ref{sec:red3}. A strong point in Proposition \ref{prop:FP} is that in \eqref{for:fp}, the time $t$ is allowed to be dependent of $j$ and $u.$ Our analysis of Sections \ref{sec:approx} and \ref{sec:4} falls however short of proving \eqref{for:fp}; it shows that for bound \eqref{for:fp} to hold at time $t = O(j 2^{j}),$ it would be sufficient to prove bounds in $O(t^{q/2})$ for the correctors $S_q.$

 \section{Appendix: symbols and operators} \label{sec:symbols}

For $m \in \R,$ the class $S^m = S^m_{1,0}$ of classical symbols is the set of all $a: \R^d \times \R^d \to \C^{n \times n }$ such that, for all $\a,\b \in \N^d,$ for some $C_{\a\b} > 0,$ for all $(x,\xi) \in \R^{2d},$ 
$$ |\d_x^\a \d_\xi^\b a(x,\xi)| \leq C_{\a\b} (1 + |\xi|^2)^{(m-|\b|)/2}.$$ %
Given a symbol $a \in S^m,$ we denote $\| a \|_r$ the norm (the order $m$ is implicit)
 \be \label{symb:norms} \| a \|_r := \sup_{|\a| + |\b| \leq r + 2([d/2] + 1)} \sup_{(x,\xi) \in \R^{2d}} \,(1 + |\xi|^2)^{(|\b| - m)/2} |\d_x^\a \d_\xi^\b a(x,\xi)|.\ee
The associated operators are, in semiclassical quantization
$$ %
 (\op_\e(a) u)(x) = \int_{\R^d} e^{i x \cdot \xi} a(x, \e \xi) \hat u(\xi) \, d\xi, \qquad \e > 0,
$$  %
and in Weyl quantization 
  \be \label{def:weyl}
  (a^{\rm w} u)(x) = \int_{\R^d \times \R^d} e^{i (x - y) \cdot \xi} a\Big( \frac{x+y}{2}, \xi\Big) u(y) \, d\xi \, dy.
  \ee
 When $\e =1,$ we denote $\op_1(a) = \op(a).$ The semiclassical Sobolev norms $\| \cdot \|_{\e,s}$ are%
 \be \label{semicl} \| u \|_{\e,s} = |(1 + |\e\xi|^2)^{s/2} \hat u|_{L^2(\R^d_\xi)}.\ee
When $\e = 1,$ the norm $\| \cdot \|_{1,s}$ is the classical $H^s$ norm. %
The Calder\'on-Vaillancourt theorem (see for instance \cite{CV,Hw}) asserts that if $a$ belongs to $S^m,$ then $\op_\e(a)$ extends to a linear bounded operator $H^m \to L^2,$ with norm controlled by $\|a \|_0:$  
 \be \label{action}
  |\op_\e(a) u|_{L^2} \lesssim \| a \|_0 \| u \|_{\e,m}, \qquad \mbox{for all $a \in S^m,$ all $u \in H^m,$}%
  \ee
the implicit constant depending only on $d.$ The same holds true in Weyl quantization (see for instance \cite{Bou}, Theorem 1.2):
\be \label{continuity:w}
 | a^{\rm w} u|_{L^2} \lesssim \| a \|_0 |u|_{H^m}, \qquad \mbox{for all $a \in S^m,$ all $u \in H^m.$}
 \ee
Stability by composition is expressed by the equality 
 \be \label{compo:e}
 \op_\e(a_1) \op_\e(a_2) = \sum_{0 \leq q \leq q_0} \e^q \op_\e(a_1 \sharp_q a_2) + \e^{q_0+1} \op_\e(R_{q_0+1}(a_1,a_2)),
 \ee
 where
 $$ a_1 \sharp_q a_2 = \sum_{|\a| = q} \frac{(-i)^{|\a|}}{\a!} \d_\xi^\a a_1 \d_x^\a a_2,$$
 and $R_{q_0+1}(a_1,a_2) \in S^{m_1 + m_2 - (q_0+1)}$ satisfies 
 \be \label{composition:e}
 \| \op_\e(R_{q_0+1}(a_1,a_2)) \|_r \lesssim \| a_1 \|_{q_0 + C(d)} \| a_2 \|_{q_0 + C(d)},
 \ee
 with $C(d) > 0$ depending only on $d.$ A composition result in classical quantization is given in Theorems 1.1.5 and 1.1.20, and Lemma 4.1.2 and Remark 4.1.4 of \cite{Le}. From there \eqref{compo:e}-\eqref{composition:e} is easily deduced by introduction of the dilations $(h_\e)$ such that $(h_\e u)(x) = \e^{d/2} u(\e x),$ and the observation that $|h_\e u|_{H^s} = \| u \|_{\e,s}$ and $\op_\e(a) = h_\e^{-1} \op(\tilde a) h_\e,$ with $\tilde a(x,\xi) := a(\e x, \xi).$

  Specializing to $a_1 = (1 + |\xi|^2)^{s/2},$ the composition result and the $H^m \to L^2$ continuity result give continuity of $\op_\e(a)$ as an operator from $H^{s+m}$ to $H^s:$
 \be \label{continuite:s}
  \| \op_\e(a) u \|_{\e,s} \lesssim \big( \| a \|_0 + \e \| a \|_{C(d)}\big) \| u \|_{\e,s+m}.
 \ee 

In Weyl quantization, there holds (see for instance Section 2.1.5 in \cite{Le})
\be \label{compo:w} 
 a_1^{\rm w} a_2^{\rm w} = \sum_{0 \leq k \leq q_0} ( a_1 \diamond_k a_2)^{\rm w} + R_{q_0+1}^{\rm w}(a_1,a_2),
 \ee
 where%
 \be \label{diamond}
 a_1 \diamond_k a_2 := \sum_{|\a| + |\b| = k} \frac{(-1)^{|\a|} (-i)^{|\a| + |\b|}}{\a!\b!} \d_x^\a \d_\xi^\b  a_1 \d_x^\b  \d_\xi^\a a_2,
 \ee
and, for all $u \in H^{m_1 + m_2 - q_0 - 1}:$ 
 \be \label{composition:w}
  | R_{q_0+1}^{\rm w}(a_1,a_2) u |_{L^2} \lesssim\| a_1 \|_{q_0  + C(d)}\| a_2 \|_{q_0  + C(d)}  |u|_{H^{m_1 + m_2 - q_0  -1}}.
  \ee

From \eqref{continuity:w} and \eqref{compo:w}-\eqref{composition:w} we deduce
\be \label{sobolev:w}
 | a^{\rm w} u|_{H^s} \lesssim \| a \|_{C(d)} |u|_{H^{s+m}}, \qquad a \in S^m, \,\, u \in H^{s+m}, \,\, s, m \in \R,
\ee
often used with $s = -m/2,$ in which case $a^{\rm w}$ appears a continuous $H^{m/2} \to H^{-m/2}$ operator.

Finally, in Section \ref{sec:insta}, we use the pointwise bound%
\be \label{sup:pw}
  |\op_\e(a) u|_{L^\infty} \lesssim \| a \|_{C(d)} |u|_{L^\infty} \Big(1 + |\ln \e| + \ln\Big(\frac{\|u\|_{\e,d/2 + m + \eta}}{|u|_{L^\infty}}\Big)\Big),
 \ee 
where $a \in S^m,$ $C(d) > 0$ depends only on $d,$ $\eta > 0$ is arbitrary, $\e \in (0,1),$ $u \in H^{d/2 + m + \eta}.$ The implicit constant in \eqref{sup:pw} depends only on $d$ and $\eta.$ Bound \eqref{sup:pw} is easily derived from estimate (B.1.1) in Appendix B of \cite{Tay} by introduction of dilations and weighted norms, as mentioned above for the composition result.

\end{document}